\documentclass[11 pt, onecolumn, oneside, a4paper, reqno]{amsart}
\usepackage{geometry}
\usepackage{amssymb}
\usepackage{mathrsfs}
\usepackage{mathtools}
\usepackage{amsmath}
\usepackage{tikz}
\usepackage{pifont}
\usepackage{color}
\usepackage[all]{xy}
\usepackage{type1cm}
\usepackage{hyperref}
\usepackage{color}
\hypersetup{colorlinks=true,
     breaklinks=true,
     linkcolor=blue,
      pageanchor=true}
\newcommand{\xiaowuhao}{\fontsize{9pt}{\baselineskip}\selectfont}

\newtheorem{thm}{Theorem}[section]

\newtheorem{cor}[thm]{Corollary}
\newtheorem{lem}[thm]{Lemma}
\newtheorem{prop}[thm]{Proposition}

\theoremstyle{definition}
\newtheorem{defn}[thm]{Definition}
\newtheorem{exam}[thm]{Example}

\newcommand{\lra}{\longrightarrow}

\newcommand{\ra}{\rightarrow}

\newcommand{\D}{\mathcal D}

\newcommand{\K}{\mathcal K}

\renewcommand{\S}{\mathcal S}
\newcommand{\R}{\mathcal R}
\newcommand{\T}{\mathcal T}
\newcommand{\U}{\mathcal U}
\newcommand{\V}{\mathcal V}
\newcommand{\W}{\mathcal W}
\newcommand{\X}{\mathcal X}
\newcommand{\Y}{\mathcal Y}

\DeclareMathOperator{\Hom}{\mathsf{Hom}}

\DeclareMathOperator*{\Mod}{\!-\mathsf{Mod}}

\DeclareMathOperator*{\per}{\mathsf{per}}
\DeclareMathOperator*{\proj}{\!-\mathsf{proj}}
\DeclareMathOperator*{\Proj}{\!-\mathsf{Proj}}

\DeclareMathOperator*{\smod}{\!-\mathsf{mod}}
\DeclareMathOperator*{\coh}{\!-\mathsf{coh}}
\DeclareMathOperator*{\qcoh}{\!-\mathsf{qcoh}}
\title[Localization theorems]{Localization theorems for weakly approximable triangulated categories}

\author[Sun-Zhang-Zhang]{Yongliang Sun, Jinbi Zhang and Yaohua Zhang*}

\address{\normalfont{Yongliang Sun \\School of Mathematics and Physics, Yancheng Institute of Technology, 224003 Jiangsu, People's Republic of China}}
\email{syl13536@126.com}

\address{\normalfont{Jinbi Zhang\\
School of Mathematical Sciences, Anhui University, 230601 Hefei, People's Republic of China}}
\email{zhangjb@ahu.edu.cn}

\address{\normalfont{Yaohua Zhang \\ Hubei Key Laboratory of Applied Mathematics, Faculty of Mathematics and Statistics, Hubei University, 430062 Wuhan, People's Republic of China}}
\email{yhzhang@hubu.edu.cn}

\thanks{* Corresponding author.}



\begin{document}
\begin{abstract}
Weakly approximable triangulated categories, introduced by Neeman, provide a powerful framework for studying localization phenomena in triangulated categories. 
In this paper, we establish new localization theorems showing that, under mild assumptions, a recollement of weakly approximable triangulated categories induces short exact sequences on several natural triangulated subcategories as well as on the associated (big) singularity categories. 
As applications, we illustrate our results in the derived categories of rings, DG algebras, and schemes.
\end{abstract}
\maketitle
\section{Introduction}
Localization and gluing phenomena play a central role in the study of triangulated categories. 
One of the most powerful categorical frameworks capturing such phenomena is provided by recollements of triangulated categories, introduced by Beilinson, Bernstein, and Deligne in \cite{BBD}. 
Recollements have since become an important tool in many areas of mathematics, including algebraic geometry, representation theory, and homological algebra. 

At its heart, a recollement of triangulated categories is a diagram $$\quad\xymatrix{\R\ar^-{i_*=i_!}[r]
&\T\ar^-{j^!=j^*}[r]\ar^-{i^!}@/^1.2pc/[l]\ar_-{i^*}@/_1.6pc/[l]
&\S\ar^-{j_*}@/^1.2pc/[l]\ar_-{j_!}@/_1.6pc/[l]}$$
that involves three interconnected short exact sequences (see Section \ref{sect-pre-rec}) satisfying certain criteria.
Such structures can be viewed as categorical analogues of short exact sequences and play a fundamental role in localization theory.

A fundamental question in the study of recollements is how they interact with natural subcategories of interest. 
A landmark result in this direction is Thomason’s localization theorem \cite{TT90}, later developed and generalized by Neeman \cite{N92} in the setting of compactly generated triangulated categories. 
Neeman proved that if the above diagram is a recollement of compactly generated triangulated categories, then the first row restricts to a short exact sequence up to direct summands on the subcategories of compact objects
$$
\S^c\stackrel{j_!}{\longrightarrow}\T^c\stackrel{i^*}{\longrightarrow}\R^c.
$$
Recently, Jin, Yang, and Zhou \cite{JYZ23} established a refinement of this result in the context of algebraically compactly generated triangulated categories.
They showed that if the first row restricts to a short exact sequence
$$\V\stackrel{j_!}{\longrightarrow}\W\stackrel{i^*}{\longrightarrow}\U,$$
where $\V, \W$ and $\U$ are triangulated subcategories of $\S, \T$ and $\R$, respectively, containing the compact objects, then this recollement induces a short exact sequence of Verdier quotient categories
$$\V/\S^c\stackrel{j_!}{\longrightarrow}\W/\T^c\stackrel{i^*}{\longrightarrow}\U/\R^c.$$
However, the methods in \cite{JYZ23} rely heavily on algebraic structures and therefore do not apply to more general triangulated categories arising in topology.

Our work advances this line of investigation by employing the theory of weakly approximable triangulated categories developed by Neeman \cite{N18a, N21b}. 
This flexible framework encompasses many classical examples, including derived categories of rings and schemes as well as homotopy categories of spectra.
This framework has led to proofs of several important conjectures and new proofs of fundamental theorems, including Bondal–Van den Bergh's conjecture \cite{N21a}, a generalization of a theorem of Rouquier \cite{N21a}, Rickard’s theorem on derived equivalences \cite{CHNS, N18c, R88}, Serre’s GAGA theorem \cite{N18a}, and results on the existence of bounded $t$-structures \cite{BCPRZ24, N22a}.

In this paper, we establish new localization results showing that recollements induce short exact sequences on several important subcategories, as well as on the (big) singularity categories.
To state our results precisely, we briefly introduce the relevant terminology.
Let $\T$ be a pre-approximable triangulated category (see Definition~\ref{def:appro}) with a compact generator $G$. The object $G$ generates a $t$-structure $(\T^{\leq 0}, \T^{\geq 0})$ of $\T$. Associated with this $t$-structure are several thick subcategories of $\T$, namely $\T^-, \T^+, \T^b,\T^{\rm sb}, \T^-_c$, $\T^b_c$, together with the singularity category $\T_{sg}$ and the big singularity category $\T^{\rm big}_{sg}$ (see Subsection~\ref{subsec:subcat}). These categories are independent of the choice of the generator. In the case of derived categories of finite-dimensional algebras, they correspond respectively to $\D^-(A)$, $\D^+(A)$, $\D^b(A)$, $\K^b(A\Proj)$, $\D^-(A\smod)$, $\D^b(A\smod)$, $\D_{sg}(A)$ and $\D_{sg}^{\rm big}(A)$. 

\begin{thm}\textnormal{(Theorem~\ref{thm:exact sequences of quotient})}\label{thm:A}
Let $\R, \S$, and $\T$ be triangulated categories that form a recollement of triangulated categories
$$\quad\xymatrix{\R\ar^-{i_*=i_!}[r]
&\T\ar^-{j^!=j^*}[r]\ar^-{i^!}@/^1.2pc/[l]\ar_-{i^*}@/_1.6pc/[l]
&\S.\ar^-{j_*}@/^1.2pc/[l]\ar_-{j_!}@/_1.6pc/[l]}$$
Then the following statements hold.
\begin{enumerate}
 \item Suppose that $\R, \S$ and $\T$ are pre-approximable. Let $\U,\V$ and $\W$ be triangulated subcategories of $\R, \S$ and $\T$, respectively, such that ${\R}^c\subseteq\U$, ${\S}^c\subseteq\V\subseteq {\S}^-_c$ and $\T^c\subseteq \W$. If the first row of the recollement is restricted to a short exact sequence
$$\V\stackrel{j_!}{\longrightarrow}\W\stackrel{i^*}{\longrightarrow}\U,$$
then it induces a short exact sequence
$$\V/{\S}^c\stackrel{\overline{j_!}}{\longrightarrow}\W/\T^c\stackrel{\overline{i^*}}{\longrightarrow}\U/{\R}^c.$$
\item Suppose that $\R,\T, \S$ are weakly approximable and that $j^*(G_{\mathcal{T}})\in \mathcal{S}_c^-$, where $G_{\mathcal{T}}$ denotes a compact generator of $\T$.
Then the first row of the recollement induces a short exact sequence $${\S}^-_c/{\S}^c\stackrel{\overline{j_!}}{\longrightarrow}{\T}^-_c/\T^c\stackrel{\overline{i^*}}{\longrightarrow}{\R}^-_c/{\R}^c.$$
In particular, this holds if $\R,\T, \S$ are locally Hom-finite approximable $R$-linear triangulated categories, with $R$ a commutative noetherian ring.  
\end{enumerate}
\end{thm}

A pre-approximable triangulated category is a triangulated category that admits a compact generator $G$ satisfying $\Hom_{\T}(G, G[i])=0$ for sufficiently large $i$. There are many examples of such categories, including the derived categories of upper-bounded DG algebras. Comparing Theorem \ref{thm:A}(1) with Jin-Yang-Zhou's theorem, we see that while their result excels when all three categories involved are algebraic, it falls short in the context of topological triangulated categories (e.g., the homotopy category of spectra), which presents an advantage for our framework. 

According to \cite{N18c}, under mild conditions, the subcategories $\T^c$ and $\T^b_c$ are mutually determining. This insight prompts us to consider the existence of counterparts for $\T^b_c$ pertaining to Neeman's localization theorem for $\T^c$. The following theorem materializes this observation.

\begin{thm}\textnormal{(Theorem~\ref{thm:exact seq Tbc})}\label{thm:B}
Let  $\R, \S$ and $\T$ be triangulated categories and admit a recollement of triangulated categories
$$\xymatrix{\R\ar^-{i_*=i_!}[r]
&\T\ar^-{j^!=j^*}[r]\ar^-{i^!}@/^1.2pc/[l]\ar_-{i^*}@/_1.6pc/[l]
&\S.\ar^-{j_*}@/^1.2pc/[l]\ar_-{j_!}@/_1.6pc/[l]}$$
Suppose that  $\R, \S$ and $\T$ are weakly approximable and coherent, and that $j^*(G_{\T})\in \S_c^-$ , where  $G_{\T}$ denotes a compact generator of $\T$.  
\begin{enumerate}
\item The second row in the recollement is restricted to a short exact sequence up to direct summands
$$\R^b_c\stackrel{i_*}\longrightarrow \T_c^b\stackrel{j^*}\longrightarrow \S^b_c.$$
\item Let $\U,\V$ and $\W$ be triangulated subcategories of $\R, \S$ and $\T$, respectively, such that ${\R}^b_c\subseteq\U\subseteq {\R}^-_c$, ${\S}^b_c\subseteq\V$ and $\T^b_c\subseteq \W$. If the second row of the recollement is restricted to a short exact sequence
$$\U\stackrel{i_*}{\longrightarrow}\W\stackrel{j^*}{\longrightarrow}\V,$$
then it induces a short exact sequence
$$\U/{\R}^b_c\stackrel{\overline{i_*}}{\longrightarrow}\W/\T^b_c\stackrel{\overline{j^*}}{\longrightarrow}\V/{\S}^b_c.$$
\item The second row induces a short exact sequence
$${\R}^-_c/{\R}^b_c\stackrel{\overline{i_*}}{\longrightarrow}\T^-_c/\T^b_c\stackrel{\overline{j^*}}{\longrightarrow}{\S}^-_c/{\S}^b_c.$$
\end{enumerate}
\end{thm}

The above two theorems yield a direct corollary.

\begin{cor}\textnormal{(Corollary~\ref{cor:for Tbc})}\label{cor:C}
Let $R$ be a commutative noetherian ring, let $\R, \S$ and $\T$ be locally Hom-finite coherent approximable $R$-linear triangulated categories and admit a recollement of triangulated categories
$$\quad\xymatrix{\R\ar^-{i_*=i_!}[r]
&\T\ar^-{j^!=j^*}[r]\ar^-{i^!}@/^1.2pc/[l]\ar_-{i^*}@/_1.6pc/[l]
&\S.\ar^-{j_*}@/^1.2pc/[l]\ar_-{j_!}@/_1.6pc/[l]}$$
Suppose that $\T$ has a compact generator $G_{\T}$ such that there is an integer $N$ with $\T(G_{\T}, G_{\T}[n])=0, n<N$, and $j^*(G_{\T})\in \S^c$(equivalently, the recollement extends one step downwards). Then the recollement induces a commutative diagram
$$\xymatrix{
\R_{sg}\ar[r]^{\overline{i_*}}\ar@{^(->}[d]&\T_{sg}\ar[r]^{\overline{j^*}}\ar@{^(->}[d]&\S_{sg}\ar@{^(->}[d]
\\
\R^-_c/\R^c\ar[r]^{\overline{i_*}}\ar[d]&\T^-_c/\T^c\ar[r]^{\overline{j^*}}\ar[d]&\S^-_c/\S^c\ar[d]
\\
\R^-_c/\R^b_c\ar[r]^{\overline{i_*}}&\T^-_c/\T^b_c\ar[r]^{\overline{j^*}}&\S^-_c/\S^b_c}$$
in which all rows and columns are short exact sequences of quotient categories.
\end{cor}

We further extend the localization theorems to the subcategories of strongly bounded objects in weakly approximable triangulated categories.
\begin{thm}\label{main-thm-l-sp}
\textnormal{(Theorem~\ref{thm-l-sp})}
Let  $\R, \S$ and $\T$ be weakly approximable triangulated categories admitting a recollement of triangulated categories
$$\xymatrix{\R\ar^-{i_*=i_!}[r]
&\T\ar^-{j^!=j^*}[r]\ar^-{i^!}@/^1.2pc/[l]\ar_-{i^*}@/_1.6pc/[l]
&\S.\ar^-{j_*}@/^1.2pc/[l]\ar_-{j_!}@/_1.6pc/[l]}$$
Then the following statements hold.
\begin{enumerate}
\item The first row of recollement is restricted to a short exact sequence up to direct summands
\begin{align*}
\xymatrixcolsep{2pc}\xymatrix{
{\mathcal{S}}^{{\rm sb}} \ar[r]^{j_!} &\mathcal{T}^{{\rm sb}} \ar[r]^{i^*}  &{\mathcal{R}}^{{\rm sb}}.
}
\end{align*}

\item Let $\mathcal{U}\subseteq \R,\mathcal{V}\subseteq \S$ and $\mathcal{W}\subseteq \T$ be triangulated subcategories satisfying ${\mathcal{R}}^{{\rm sb}} \subseteq\mathcal{U}$, ${\mathcal{S}}^{{\rm sb}}\subseteq\mathcal{V}\subseteq \mathcal{S}^-$ and $\mathcal{T}^{{\rm sb}}\subseteq \mathcal{W}$. 
If the first row of the recollement is restricted to a short exact sequence
\begin{align*}
\xymatrixcolsep{2pc}\xymatrix{
\mathcal{V} \ar[r]^{j_!} &\mathcal{W} \ar[r]^{i^*}  &\mathcal{U},
}
\end{align*}
then it induces a short exact sequence
\begin{align*}
\xymatrixcolsep{2pc}\xymatrix{
\mathcal{V}/\mathcal{S}^{{\rm sb}}\ar[r]^{j_!} &\mathcal{W}/\mathcal{T}^{{\rm sb}} \ar[r]^{i^*}  &\mathcal{U}/\mathcal{R}^{{\rm sb}}.
}
\end{align*}

\item The first row induces a short exact sequence
\begin{align*}
\xymatrixcolsep{2pc}\xymatrix{
\mathcal{S}^-/\mathcal{S}^{{\rm sb}}\ar[r]^{j_!} &\mathcal{T}^-/\mathcal{T}^{{\rm sb}} \ar[r]^{i^*}  &\mathcal{R}^-/\mathcal{R}^{{\rm sb}}.
}
\end{align*}
\end{enumerate}
\end{thm}

We also extend the localization theorems to the context of  bounded subcategories in weakly approximable triangulated categories.
\begin{thm}\label{main-thm-l-sb}
\textnormal{(Theorem~\ref{thm-l-sb})}
Let  $\R, \S$ and $\T$ be weakly approximable triangulated categories admitting a recollement of triangulated categories
$$\quad\xymatrix{\R\ar^-{i_*=i_!}[r]
&\T\ar^-{j^!=j^*}[r]\ar^-{i^!}@/^1.2pc/[l]\ar_-{i^*}@/_1.6pc/[l]
&\S.\ar^-{j_*}@/^1.2pc/[l]\ar_-{j_!}@/_1.6pc/[l]}$$
Then the following statements hold.
\begin{enumerate}
\item The second row in the recollement is restricted to a short exact sequence up to direct summands
\begin{align*}
\xymatrixcolsep{2pc}\xymatrix{\mathcal{R}^b \ar[r]^{i_*} &\mathcal{T}^b \ar[r]^{j^*}  &\mathcal{S}^b. 
}
\end{align*}

\item Let $\mathcal{U}\subseteq \R,\mathcal{V}\subseteq \S$ and $\mathcal{W}\subseteq \T$ be triangulated subcategories satisfying ${\mathcal{R}}^b \subseteq\mathcal{U} \subseteq {\mathcal{R}}^-$ (or ${\mathcal{R}}^b \subseteq\mathcal{U} \subseteq {\mathcal{R}}^+$), ${\mathcal{S}}^b\subseteq\mathcal{V}$ and $\mathcal{T}^b\subseteq \mathcal{W}$. 
If the second row of the recollement is restricted to a short exact sequence
\begin{align*}
\xymatrixcolsep{2pc}\xymatrix{
\mathcal{U} \ar[r]^{i_*} &\mathcal{W} \ar[r]^{j^*}  &\mathcal{V},
}
\end{align*}
then it induces a short exact sequence
\begin{align*}
\xymatrixcolsep{2pc}\xymatrix{
\mathcal{U}/\mathcal{R}^b  \ar[r]^{i_*} &\mathcal{W}/\mathcal{T}^b \ar[r]^{j^*}  &\mathcal{V}/\mathcal{S}^b. 
}
\end{align*}

\item  The second row induces short exact sequences
\begin{align*}
\xymatrixcolsep{2pc}\xymatrix{
\mathcal{R}^-/\mathcal{R}^b  \ar[r]^{i_*} &\mathcal{T}^-/\mathcal{T}^b \ar[r]^{j^*}  &\mathcal{S}^-/\mathcal{S}^b, \text{ and}\\
\mathcal{R}^+/\mathcal{R}^b  \ar[r]^{i_*} &\mathcal{T}^+/\mathcal{T}^b \ar[r]^{j^*}  &\mathcal{S}^+/\mathcal{S}^b. 
}
\end{align*}
\end{enumerate}
\end{thm}

Theorem \ref{main-thm-l-sp} and Theorem \ref{main-thm-l-sb} naturally lead to the following corollary, which provides a commutative diagram of short exact sequences.
\begin{cor}\label{main-cor-3-3-pb-}
\textnormal{(Corollary~\ref{cor-3-3-pb-})}
Let  $\R, \S$ and $\T$ be weakly approximable triangulated categories admitting a recollement of triangulated categories
$$\quad\xymatrix{\R\ar^-{i_*=i_!}[r]
&\T\ar^-{j^!=j^*}[r]\ar^-{i^!}@/^1.2pc/[l]\ar_-{i^*}@/_1.6pc/[l]
&\S.\ar^-{j_*}@/^1.2pc/[l]\ar_-{j_!}@/_1.6pc/[l]}$$
Suppose that $\T$ has a compact generator $G_{\T}$ such that there is an integer $N$ with $\T(G_{\T}, G_{\T}[n])=0$ for all $ n<N$, and $j^*(G_{\T})\in \S^{\rm sb}$(e.g. the recollement extends one step downwards). Then the recollement induces a commutative diagram
$$\xymatrix{
\mathcal{R}^{\rm big}_{sg}\ar[r]^{\overline{i_*}}\ar@{^(->}[d]&\mathcal{T}^{\rm big}_{sg}\ar[r]^{\overline{j^*}}\ar@{^(->}[d]&\mathcal{S}^{\rm big}_{sg}\ar@{^(->}[d]\\
\mathcal{R}^-/\mathcal{R}^{{\rm sb}}\ar[r]^{\overline{i_*}}\ar[d]&\mathcal{T}^-/\mathcal{T}^{{\rm sb}}\ar[r]^{\overline{j^*}}\ar[d]&\mathcal{S}^-/\mathcal{S}^{{\rm sb}}\ar[d]\\
\mathcal{R}^-/\mathcal{R}^b\ar[r]^{\overline{i_*}}&\mathcal{T}^-/\mathcal{T}^b\ar[r]^{\overline{j^*}}&\mathcal{S}^-/\mathcal{S}^b
}$$
in which all horizontal and vertical sequences are short exact sequences.
\end{cor}

Now, we apply Corollary \ref{cor:C} and Corollary \ref{main-cor-3-3-pb-} to derived categories of schemes. 
\begin{cor}\label{main-cor:scheme}
{\rm (Corollary \ref{cor:scheme})}
Let $X$ be a quasi-compact and separated scheme. Assume that $Z\subseteq X$ is a closed subset with quasi-compact complement $U$.
Then the following statements hold. 
\begin{enumerate}
\item There is a short exact sequence of big singularity categories
 $$\D_{sg,Z}^{\rm big}(X) \longrightarrow \D_{sg}^{\rm big}(X) \longrightarrow \D_{sg}^{\rm big}(U).$$

\item
If, in addition, $X$ is noetherian, then there is a short exact sequence of singularity categories
$$\D_{sg,Z}(X)\longrightarrow \D_{sg}(X)\longrightarrow \D_{sg}(U).$$

\end{enumerate}
\end{cor}

The contents of this paper are outlined as follows. In Section~\ref{sec:pre}, we fix notation
and recall some definitions and basic facts used throughout the paper. In particular, we recall the definitions of recollements and weakly approximable triangulated categories, and technical lemmas needed in proving the main theorems. Furthermore, we show some restriction results of functors and recollements in the framework of weakly approximable triangulated categories.  In Section~\ref{sec:loc theorem}, we prove Theorems~\ref{thm:A}, \ref{thm:B}, \ref{main-thm-l-sp}, and \ref{main-thm-l-sb}, as well as Corollaries~\ref{cor:C} and \ref{main-cor-3-3-pb-}. In Section~\ref{sec:applications}, we apply the localization theorems and corollaries in Section~\ref{sec:loc theorem} to some classical situations, such as the derived categories of rings, the derived categories of DG algebras and the derived categories of schemes.

\section{Preliminaries}\label{sec:pre}
In this section, we recall some notation, definitions, and basic facts that will be used throughout the paper.
\subsection{Notation}
Throughout, $R$ denotes a commutative noetherian ring. Let $\T$ be a triangulated category. We abbreviate the Hom-set $\Hom_{\T}(X, Y)$ by ${\T}(X, Y)$, and ${\T}(X, Y[<n])=0$ by the Hom-sets ${\T}(X, Y[i])=0$ for $i<n$. Let $\R, \S$ be two subcategories of $\T$, we define
$$\R*\S:=\{T\in\T\mid T~\text{admits a triangle}~R\to T\to S\to R[1], R\in\R, S\in\S\}.$$
Assume $\T$ has small coproducts and $G$ is an object of $\T$. Let $a, b\in\mathbb{Z}\cup\{-\infty, +\infty\}$.
We denote by $\overline{\langle G\rangle}^{[a,b]}$ the smallest subcategory of $\T$ which contains $G[-i]$ where $a\leq i\leq b$ and is closed under direct summands, coproducts and extensions. Let $n$ be a positive integer, the notation $\overline{\langle G\rangle}_n^{[a,b]}$ is defined inductively as
\begin{align*}
    \overline{\langle G\rangle}_1^{[a,b]}&=\text{direct summands of coproducts of objects in}~\{G[-i]\mid a\leq i\leq b\},\\
    \overline{\langle G\rangle}_n^{[a,b]}&=\text{direct summands of objects in}~\overline{\langle G\rangle}_1^{[a,b]}*\overline{\langle G\rangle}_{n-1}^{[a,b]}.
\end{align*}

Let $A$ be a ring. Denote by $A\Mod$, $A\smod$, $A\Proj$ and $A\proj$ the categories of all left $A$-modules, finitely presented $A$-modules, all projective  $A$-modules, and finitely generated projective $A$-modules, respectively.
Let $\D(A)$ be the derived category of $A\Mod$. 
We write $\D^b(A)$ and $\D^-(A)$ for the bounded and upper bounded derived categories of $A\Mod$. If $A$ is coherent, we write $\D^b(A\smod)$ and $\D^-(A\smod)$ for the bounded and upper bounded derived categories of $A\smod$.
We denote by $\K^b(A\Proj)$ and $\K^b(A\proj)$ the homotopy categories of bounded complexes of objects in $A\Proj$ and $A\proj$, respectively. We denote by $\K^{-,b}(A\Proj)$ and $\K^{-,b}(A\proj)$ the homotopy categories of upper bounded complexes of objects in $A\Proj$ and $A\proj$, respectively.

Let $X$ be a scheme and $Z\subseteq X$ be  a closed subset. We denote by $\D_{qc}(X)$ the category of all complexes of $\mathcal{O}_X$-modules with quasicoherent cohomology, and by $\D_{qc,Z}(X)$ the full subcategory of $\D_{qc}(X)$ consisting of complexes with quasicoherent cohomology supported on $Z$. Furthermore, suppose that $X$ is quasicompact and quasiseparated, and that the open set $X \setminus Z$ is quasicompact. We denote by $\D^-_{qc,Z}(X)$ (resp. $\D^+_{qc,Z}(X)$, and $\D^b_{qc,Z}(X)$) the full subcategory of $\D_{qc,Z}(X)$ consisting of complexes that are bounded above (resp. bounded below, and bounded), and by
$\D_{Z}^\mathsf{per}(X)$ the full subcategory of $\D_{qc,Z}(X)$ consisting of all perfect complexes supported on $Z$. Here, a complex is \textit{perfect} if it is locally isomorphic to a bounded complex of finite-rank vector bundles. 
If $X$ is noetherian, we further denote by $\D^-_{coh,Z}(X)$ (resp. $\D^b_{coh,Z}(X)$) the full subcategory of $\D^-_{qc,Z}(X)$ (resp. $\D^b_{qc,Z}(X)$) consisting of complexes whose cohomology sheaves are coherent and supported on $Z$. 
When $Z=X$, the support condition is vacuous, and we suppress the subscript $Z$ from the notation (for example, $\D^-_{qc}(X):=\D^-_{qc,X}(X)$).

\subsection{$t$-structures, short exact sequences and recollements}\label{sect-pre-rec}
Let $\T$ be a triangulated category. A pair $(\T^{\leq 0}, \T^{\geq 0})$ of subcategories of $\T$ is called a
{\em $t$-structure} (\cite{BBD}) if it satisfies the following conditions:
\begin{enumerate}
  \item $\T^{\leq 0}[1]\subseteq \T^{\leq 0}$ and $\T^{\geq 0}[-1]\subseteq \T^{\geq 0}$;
  \item $\T(\T^{\leq 0}, \T^{\geq 0}[-1])=0$;
  \item For any $T\in\T$, there exists a triangle
  $$U\longrightarrow T\longrightarrow V\longrightarrow U[1],$$
 where $U\in\T^{\leq 0}$ and $V\in\T^{\geq 0}[-1]$.
\end{enumerate}

For each $n\in \mathbb{Z}$, we denote $\T^{\leq n}:=\T^{\leq 0}[-n]$ and $\T^{\geq n}:=\T^{\geq 0}[-n]$. 
Two $t$-structures $(\T_1^{\le 0},\T_1^{\ge 0})$  and $(\T_2^{\le 0},\T_2^{\ge 0})$ on $\T$ are said to be \textit{equivalent} if there exists a natural number $n$ such that $\T_1^{\le -n}\subseteq \T_2^{\le 0}\subseteq \T_1^{\le n}$. It is straightforward to verify that this defines an equivalence relation.

Let $\mathcal{T}$ be a triangulated category with arbitrary coproducts, and let $G$ be a compact generator of $\mathcal{T}$ . According to \cite[Theorem A.1]{ALS03} or its generalization \cite[Theorem 2.3]{N18d}, the pair
$$(\T^{\leq 0}_G, \T^{\geq 0}_G):=(\overline{\langle G\rangle}^{(-\infty, 0]}, (\overline{\langle G\rangle}^{(-\infty, -1]})^{\perp})$$
is a $t$-structure on $\mathcal{T}$. Both $\T^{\leq 0}_G$ and $\T^{\geq 0}_G$ are closed under coproducts in $\T$.
This $t$-structure is called the $t$-structure  on $\T$ \textit{generated by $G$}. Furthermore, if $H$ is another compact object of $\mathcal{T}$, then the $t$-structure  on $\T$ generated by $G$ and by $H$ are equivalent.
The {\em preferred equivalence class of} $t$-structures on $\T$ is defined to be the equivalence class containing the $t$-structure $(\T^{\leq 0}_G, \T^{\geq 0}_G)$ generated by $G$.

\begin{defn}
A sequence of triangulated categories
    $$\R\stackrel{F}{\longrightarrow}\T\stackrel{G}{\longrightarrow}\S$$
is a {\em short exact sequence} if it satisfies
\begin{enumerate}
    \item $F$ is fully faithful;
    \item $GF=0$; and
    \item the induced functor $\overline{G}: \T/\R\to \S$ is an equivalence.
\end{enumerate}
\end{defn}
This notion is equivalent to the definition of Verdier quotient of triangulated categories. The sequence is called a {\em short exact sequence up to direct summands} if $\overline{G}$ in (3) is an equivalence up to direct summands.

\begin{defn}
Let $\T$, $\X$ and $\Y$ be triangulated categories.
We say that $\T$ is a {\em recollement} (\cite{BBD}) of $\X$
and $\Y$ if there is a diagram of six triangulated functors
$$\xymatrix{\X\ar^-{i_*=i_!}[r]&\T\ar^-{j^!=j^*}[r]
\ar^-{i^!}@/^1.2pc/[l]\ar_-{i^*}@/_1.6pc/[l]
&\Y\ar^-{j_*}@/^1.2pc/[l]\ar_-{j_!}@/_1.6pc/[l]}$$ such
that
\begin{enumerate}
    \item $(i^*,i_*),(i_!,i^!),(j_!,j^!)$ and $(j^*,j_*)$ are adjoint
pairs;
    \item $i_*,j_*$ and $j_!$ are fully faithful functors;
    \item $i^!j_*=0$; and
    \item for each object $T\in\T$, there are two triangles in
$\T$
$$
i_!i^!(T)\to T\to j_*j^*(T)\to i_!i^!(T)[1],
$$
$$
j_!j^!(T)\to T\to i_*i^*(T)\to j_!j^!(T)[1].
$$
\end{enumerate}
\end{defn}
We say that $\T$ is a {\em half recollement} of $\X$ and $\Y$ if there is a diagram
$$\xymatrix{\X\ar^-{i_*=i_!}[r]
&\T\ar^-{j^!=j^*}[r]\ar^-{i^!}@/^1.2pc/[l]
&\Y.\ar^-{j_*}@/^1.2pc/[l]}$$
that satisfies the conditions above. Indeed, a half recollement coincides with the concept of Bousfield localization \cite[Definition 9.1.1]{N01}.
In this case, the rows are exact sequences of triangulated categories (\cite[Proposition 4.9.1]{K10} and its dual).

Let $\T, \S$ be triangulated categories endowed with $t$-structures $(\T^{\leq 0}, \T^{\geq 0})$ and $(\S^{\leq 0}, \S^{\geq 0})$, respectively. A functor $\mathbf{F}:\T\to \S$ is {\em right $t$-exact} if $\mathbf{F}(\T^{\leq 0})\subseteq \S^{\leq 0}$, and {\em left $t$-exact} if $\mathbf{F}(\T^{\geq 0})\subseteq \S^{\geq 0}$, and {\em $t$-exact} if both left and right $t$-exact.

Let $\T$ be a recollement of $\X$ and $\Y$. Let $(\X^{\leq 0}, \X^{\geq 0})$ and $(\Y^{\leq 0}, \Y^{\geq 0})$ be $t$-structures of $\X$ and $\Y$, then we get a glued $t$-structure $(\T^{\leq 0}, \T^{\geq 0})$ on $\T$, where
 \begin{align*}
    \T^{\leq 0}=& \{T\in\T\mid i^{*}(T)\in \X^{\leq 0}, j^{!}(T)\in \Y^{\leq 0}\}, \\
    \T^{\geq 0}=& \{T\in\T\mid i^{!}(T)\in \X^{\geq 0}, j^{!}(T)\in \Y^{\geq 0}\}.
 \end{align*}
With respect to these $t$-structures, according to \cite[1.3.17(iii)]{BBD}, the functors $i^*$, $j_!$ are right $t$-exact, $i_*$, $j^*$ are $t$-exact and $i^!$, $j_*$ are left $t$-exact.

The following result establishes a connection between recollements and short exact sequences of triangulated categories.
It is well known; see  \cite[1.4.4, 1.4.5, 1.4.8]{BBD}, \cite[Chapter III, Lemma 1.2 (1) and Chapter IV, Proposition 1.11]{BR07} .

\begin{lem}
\label{lem-rec-es}
Let $\T$, $\X$ and $\Y$ be triangulated categories.
\begin{enumerate}
\item Assume that there is a short exact sequence of triangulated categories {\rm (}possibly up to direct summands{\rm)}
$$\xymatrix{\mathcal{X} \ar[r]^{i_*} &\mathcal{T}  \ar[r]^{j^*}  &\mathcal{Y}.
}$$
Then $i_*$ has a left adjoint {\rm (}respectively, right adjoint{\rm)} if and only if $j^*$ has a left adjoint {\rm (}respectively, right adjoint{\rm)}. In this case, $i_*$ and $j^*$ together with their left adjoints {\rm (}respectively, right adjoints{\rm)} form a left recollement {\rm (}respectively, right recollement{\rm)} of $\mathcal{T}$ in terms of $\mathcal{X}$ and $\mathcal{Y}$.
Conversely,  the two rows of a left recollement are short exact sequences of triangulated categories.

\item Assume that there is a diagram
$$\xymatrix{\X\ar^-{i_*=i_!}[r]&\T\ar^-{j^!=j^*}[r]
\ar^-{i^!}@/^1.2pc/[l]\ar_-{i^*}@/_1.6pc/[l]
&\Y\ar^-{j_*}@/^1.2pc/[l]\ar_-{j_!}@/_1.6pc/[l]}$$ such that $(i^*,i_*),(i_!,i^!),(j_!,j^!)$ and $(j^*,j_*)$ are adjoint
pairs. If any one of the three rows is a short exact sequence of triangulated categories, then the diagram is a recollement. Conversely, If it is a recollement, then all the three rows are short exact sequences of triangulated categories. 
\end{enumerate}
\end{lem}

\begin{lem}\label{lem-rec-ex-down}
Consider the following recollement of triangulated categories:
$$\xymatrix{\X\ar^-{i_*=i_!}[r]&\T\ar^-{j^!=j^*}[r]
\ar^-{i^!}@/^1.2pc/[l]\ar_-{i^*}@/_1.6pc/[l]
&\Y.\ar^-{j_*}@/^1.2pc/[l]\ar_-{j_!}@/_1.6pc/[l]}$$ 
Then the following conditions are equivalent:
\begin{enumerate}
\item $i_*$ preserves compact objects.

\item $j^*$ preserves compact objects.

\item $i^!$ admits a right adjoint.

\item $j_*$ admits a right adjoint.

\item the recollement extends one step downwards.
\end{enumerate}
In this case, the original recollement restricts to a left recollement of compact subcategories:
$$\xymatrix{\X^c\ar^-{i_*=i_!}[r]
&\T^c\ar^-{j^!=j^*}[r]
\ar_-{i^*}@/_1.6pc/[l]
&\Y^c.\ar_-{j_!}@/_1.6pc/[l]}$$ 
\end{lem}

The following lemma is a triangulated-category version of "$3\times 3$ Lemma" which will be used to prove the main theorems in Section~\ref{sec:loc theorem}.

\begin{lem}\textnormal{(\cite[Lemma 3.2]{KY16})}\label{lem:third isomorphism}
  Let 
  $$\xymatrix{
  \U\ar[r]^{i|_{\U}}\ar@{^(->}[d]&\W\ar@{^(->}[d]\ar[r]^{\mathbf{F}|_{\W}}&\V\ar@{^(->}[d]
  \\
  \R\ar@{^(->}[r]^i&\T\ar[r]^{\mathbf{F}}&\S
  }$$
be a commutative diagram of triangulated categories and triangulated functors. Suppose the first row is exact up to direct summands and the second is exact. Then the diagram can be completed into a commutative diagram
  $$\xymatrix{
  \U\ar[r]^{i|_{\U}}\ar@{^(->}[d]&\W\ar@{^(->}[d]\ar[r]^{\mathbf{F}|_{\W}}&\V\ar@{^(->}[d].
  \\
  \R\ar@{^(->}[r]^i\ar[d]&\T\ar[r]^{\mathbf{F}}\ar[d]^q&\S\ar[d]
  \\
  \R/\U\ar[r]^{\overline{i}}&\T/\W\ar[r]^{\overline{\mathbf{F}}}&\S/\V
  }$$
If moreover, $\overline{i}$ is fully faithful, then the third row is exact.
\end{lem}

\begin{proof}
By \cite[Lemma 3.2]{KY16}, $\overline{\mathbf{F}}$ is a triangulated quotient functor with kernel $\mathsf{thick}(qi(\R))$. Note that  $\overline{i}: \R/\U\to \mathsf{thick}(qi(\R))$ is fully faithful and dense up to direct summands. Thus, the third row is exact.
\end{proof}

The lemma below provides useful sufficient conditions to detect whether the induced functor $\overline{i}$ in the last lemma is fully faithful or not.

\begin{lem}\textnormal{(\cite[Lemma 4.7.1]{K10}, \cite[Lemma 10.3]{K96})}\label{fully-faithful}
 Let $\T$ be a triangulated category with two full triangulated subcategories $\R$ and $\S$. Then we put $\U=\S\cap\R$ and we can form the following commutative diagram of exact functors
$$\xymatrix{
\U\ar@{^(->}[r]\ar@{^(->}[d]&
\R\ar^{\textnormal{cano.}\quad}[r]\ar@{^(->}[d]&
\R/\U\ar_{\mathbf{J}}[d]\\
\S\ar@{^(->}[r]&
\T\ar^{\textnormal{cano.}\quad}[r]&
\T/\S
}.$$
Assume that either
\begin{enumerate}
\item  every morphism from an object in $\R$ to an object in $\S$ factors some object in $\U$,or

\item every morphism from an object in $\S$ to an object in $\R$ factors some object in $\U$.
\end{enumerate}
Then the induced functor $\mathbf{J}: \R/\U\to \T/\S$ is fully faithful.
\end{lem}

\subsection{Weakly approximable triangulated categories}
\subsubsection{Pre-approximable and approximable categories}
We briefly recall the notion of approximable triangulated categories introduced by Neeman (see \cite[Definition 0.21]{N18a}).

\begin{defn}\label{def:appro}
Let $\T$ be a triangulated category with coproducts.
\begin{enumerate}
\item The category $\T$ is called {\em pre-approximable} if it admits a compact generator $G$ such that there exists an integer $n>0$ with $\T(G, G[i])=0$ for all $i\geq n$.
\item The category $\T$ is called {\em weakly approximable} if, in addition, for every $X\in \T_G^{\le 0}$ there exists a triangle
$$E \longrightarrow X \longrightarrow D\longrightarrow E[1]$$
such that $E \in {\overline{\langle G\rangle}}^{[-n, n]}$ and $D \in \mathcal{T}^{\leq-1}_G$.
\item The category $\T$ is called \emph{approximable} if condition (ii) can be strengthened by requiring
$E \in {\overline{\langle G\rangle_n}}^{[-n, n]}$.
\end{enumerate}
\end{defn}

Indeed, in Neeman's original definition, the $t$-structure is actually in the preferred equivalence class, namely, it is equivalent to $(\mathcal{T}^{\leq 0}_G, \mathcal{T}^{\geq 0}_G)$ (\cite[Proposition 2.4]{N18a}). The definitions of (pre-, weakly) approximable triangulated categories are independent of the choices of the compact generator (\cite[Proposition 2.6]{N18a}).

According to \cite[Remark 3.3]{N18a}, if a triangulated category $\T$ has a compact generator $G$ satisfying $\T(G, G[i])=0, \forall i\geq 1$, then $\T$ is approximable. Furthermore, the following proposition provides a more useful detection.

\begin{prop}\textnormal{(\cite[Corollary 4.3]{BV20})}\label{prop:approximation gene}
Let $\T$ be a triangulated category admitting arbitrary coproducts and with a compact generator $G$.
\begin{enumerate}
    \item If $\T(G, G[>1])=0$, then $\T$ is weakly approximable.
    \item If $G=\bigoplus_{1\leq i\leq n}G_i$ with $\T(G, G[>1])=0$ and $\T(G_i, G_j[1])=0$ for $i\leq j$, then $\T$ is approximable.
\end{enumerate}
\end{prop}

\subsubsection{Intrinsic subcategories in a weakly approximable triangulated category}\label{subsec:subcat}
Let $\T$ be a triangulated category with cproducts and a compact generator $G$, and let
$(\T^{\le0},\T^{\ge0})$ be the $t$-structure generated by $G$.
Associated with this $t$-structure are several intrinsic triangulated
subcategories introduced in \cite{N18a}.
\begin{enumerate}
\item \emph{Bounded above objects}: $\mathcal{T}^-\coloneqq \cup_{m=1}^\infty\mathcal{T}^{\leq m}$;
\item \emph{Bounded below objects}: $\mathcal{T}^+\coloneqq \cup_{m=1}^\infty\mathcal{T}^{\geq-m}$;
\item \emph{Bounded objects}: $\mathcal{T}^b\coloneqq \mathcal{T}^-\cap\mathcal{T}^+$;
\item \textit{Strongly bounded objects}: $\T^{\rm sb}=\cup_{A\le B}\overline{\langle G\rangle}^{[A,B]}$;
\item \emph{Compact objects}: $\mathcal{T}^c$;
\item \emph{Pseudo-compact objects}: $\mathcal{T}^-_c$, where an object $F\in\mathcal{T}$ belongs to $\mathcal{T}^-_c$ if, for any integer $m>0$, there exists in $\mathcal{T}$ a distinguished triangle $E\ra F\ra D$ with $E\in\mathcal{T}^c$ and with $D\in\mathcal{T}^{\leq-m}$. Thus
\[
\mathcal{T}^-_c=\cap_{m=1}^\infty\big(\mathcal{T}^c*\mathcal{T}^{\leq-m}\big);
\]
\item \emph{Bounded pseudo-compact objects}: $\mathcal{T}^b_c\coloneqq \mathcal{T}^-_c\cap\mathcal{T}^b$.
\end{enumerate}

These subcategories are intrinsic, that is, they are independent of the choice of the compact generator. Moreover, $\T^-,\T^+$, $\T^b$ and $\T^{\rm sb}$ are thick subcategories.
If $\T$ is pre-approximable, then $\T^-_c$ and $\T^b_c$ are thick subcategories \cite[Proposition 0.19]{N18a}.
These subcategories have obvious relations
$$\T^{\rm sb}\subseteq \mathcal{T}^-,\;\T^c\subseteq \T^-_c\subseteq \T^-~\text{and}~\T^{b}_c\subseteq \T^-_c.$$
But, in a pre-approximable triangulated category $\T$, the subcategories $\T^b_c$ and $\T^c$ have unfixed relations. We remark that all the four cases below occur: (1) $\T^c\subseteq \T^b_c$ (Example~\ref{exam:approximable ring}), (2) $\T^c=\T^b_c$ (if the global dimension of $\Lambda$ in Example~\ref{exam:approximable ring} is finite), (3) $\T^b_c\subseteq \T^c$ (Section~\ref{subsec:DG algebra}), (4) $\T^c$ and $\T^b_c$ are not comparable (Example~\ref{exam:not comparable}).

\begin{exam}\label{exam:not comparable}
Let $A$ be a DG algebra $k[x, y]/\langle y^2\rangle$ over a field $k$ with $x$ and $y$ in degree $-1$ and with trivial differential. $\D(A)$ is a locally Hom-finite approximable $k$-linear triangulated category. In this case, $A\not\in\D_{fd}(A)=\D(A)^b_c$ and $A/(x, y)\not\in \mathsf{per}(A)=\D(A)^c$. This implies that $\D(A)^b_c$ and $\D(A)^c$ are not comparable.
\end{exam}

The following lemma provides a simple sufficient condition to ensure that $\T^c$ is contained in $\T^b_c$.

\begin{lem}\label{lem:Tc in Tbc}
  Let $\T$ be a pre-approximable triangulated category. If $\T$ admits a compact generator $G$ satisfying $\T(G, G[\ll 0])=0$, then $\T^c\subseteq \T^b_c$. 
\end{lem}
\begin{proof}
Suppose that $\T(G, G[\ll 0])=0$. Since $\T_c^b=\T_c^-\cap \T^b$ and $\T^c\subseteq \T_c^-$, we just need to prove $\T^c\subseteq \T^b$. Since $\T^c\subseteq \T^-$, it suffices to show $\T^c\subseteq \T^+$. By the hypothesis, $G\in \T^+$ and also $\T^c\subseteq \T^+$.
\end{proof}

The following lemma characterizes $\mathcal{T}^{{\rm sb}}$ equivalently and provides a sufficient condition for $\mathcal{T}^{{\rm sb}}\subseteq\mathcal{T}^b$.

\begin{lem}\label{lem-sp=big}
\textnormal{(\cite{ccz25})}
Let $\mathcal{T}$ be a weakly approximable triangulated category with a compact generator $G$. 
Then
\begin{align*}
\mathcal{T}^{{\rm sb}}&=\{Y\in \mathcal{T}^-\mid 
{\mathcal{T}}(Y,\mathcal{T}^-[\gg 0])=0 \}\\
&=\{Y\in \mathcal{T}^-\mid {\mathcal{T}}(Y,\mathcal{T}^{\le n})=0 \text{ for some } n\in \mathbb{Z}\}.
\end{align*}
If, in addition, $\T(G, G[\ll 0])=0$, then
$$\mathcal{T}^{{\rm sb}}=\{Y\in \mathcal{T}^b\mid {\mathcal{T}}(Y, \mathcal{T}^b[\gg0])=0\}\subseteq \mathcal{T}^b.$$
\end{lem}

Let $\T$ be a weakly approximable triangulated category. Suppose that $\T$ admits a compact generator $G$ satisfying $\T(G, G[\ll 0])=0$. Combining Lemma \ref{lem:Tc in Tbc} with Lemma \ref{lem-sp=big}, we obtain $\T^c\subseteq \T^b_c$ and $\mathcal{T}^{{\rm sb}}\subseteq\mathcal{T}^b$. Then the \textit{(small) singularity category} of $\T$ is defined as
$\T_{sg}:=\T^b_c/\T^c$,
while the \textit{big singularity category} of $\T$ is given by
$\T^{\rm big}_{sg}:=\T^b/\T^{\rm sb}$.

\subsubsection{Locally Hom-finite approximable triangulated categories}
In this subsection, we recall some concepts from \cite[Definition 0.1]{N18a}.
Let $\mathcal{T}$ be an $R$-linear triangulated category with coproducts. Denote by $\T^c$ the subcategory consisting of compact objects. An $R$-linear functor $\mathbf{H}: [\T^c]^{\mathsf{op}}\to R\Mod$ is an {\em $\T^c$-cohomological functor} if it takes triangles to long exact sequences. An $\T^c$-cohomological functor $\mathbf{H}$ is called {\em locally finite (finite, respectively) $\T^c$-cohomological} if for any $T\in\T^c$, (1) $\mathbf{H}(T)$ is a finite $R$-module, and (2) $\mathbf{H}(T[i])=0, i\ll 0 (|i|\gg 0, \text{respectively})$.

An $R$-linear triangulated category $\T$ is called {\em locally Hom-finite} if $\T^c$ is Hom-finite, namely the Hom-sets $\T(X, Y)$ are all finite $R$-modules for $X,Y\in \T^c$. In particular, if $\T$ has a compact generator $G$, then $\T$ is locally Hom-finite if $\T(G, G[i])$ is a finite generated $R$-module for each $i\in\mathbb{Z}$.

The following theorem gives sufficient and necessary conditions to characterize objects in the subcategories $\T^-_c$ and $\T^b_c$ in an approximable triangulated category.

\begin{thm}\label{lem:object to bc}
Let $\T$ be a locally Hom-finite approximable $R$-linear triangulated category and $X$ an object in $\T$.
Then
\begin{enumerate}
    \item \textnormal{(\cite[Theorem 7.18]{N18a})} $X\in \T^-_c$ if and only if the functor $\T(-, X)$ is locally finite $\T^c$-cohomological;
    \item \textnormal{(\cite[Theorem 7.20]{N18a})} $X\in \T^b_c$ if and only if the functor $\T(-, X)$ is finite $\T^c$-cohomological.
\end{enumerate}
\end{thm}

\subsubsection{Examples}
We present some prototypical examples of (weakly) approximable triangulated categories.
\begin{exam}\textnormal{(\cite[Example 3.1]{N18a})}\label{exam:approximable ring}
The unbounded derived category $\D(\Lambda)$ of a ring $\Lambda$ is approximable and has subcategories
\begin{align*}
\D(\Lambda)^-&=\K^-(\Lambda\Proj)\simeq\D^-(\Lambda), \D(\Lambda)^+=\D^+(\Lambda),\\ \D(\Lambda)^b&=\D^b(\Lambda),
\D(\Lambda)^{\rm sb}=\K^b(\Lambda\Proj),\\
 \D(\Lambda)^-_c&=\D^-(\Lambda\proj), \D(\Lambda)^c=\K^{b}(\Lambda\proj), \D(\Lambda)^b_c=\K^{-,b}(\Lambda\proj).
\end{align*}
Note that the big singularity category of $\mathcal{D}(\Lambda)$ coincides with the 
classical big singularity category of $\Lambda$, which is defined as the quotient category
\[
\mathcal{D}_{sg}^{\text{big}}(\Lambda) \coloneqq \mathcal{D}^b(\Lambda) / \mathcal{K}^b(\Lambda\text{-proj}),
\]
following Beligiannis \cite{Bel2000}.

If $\Lambda$ is coherent, then
$$\D(\Lambda)^b_c=\D^b(\Lambda\smod), \D(\Lambda)^-_c=\D^-(\Lambda\smod).$$
The small singularity category of $\mathcal{D}(\Lambda)$ coincides with the 
classical singularity category of $\Lambda$, defined by
\[
\mathcal{D}_{sg}(\Lambda) \coloneqq \mathcal{D}^b(\Lambda\text{-mod}) / \mathcal{K}^b(\Lambda\text{-proj}),
\]
following Buchsweitz \cite{Buch87}.
\end{exam}

\begin{exam}\label{ex-geometry}\textnormal{(\cite[Lemma 3.5]{N18a}, \cite[Remark 1.6]{N21b}, \cite[Section 8]{N22a}})
Let $X$ be a quasicompact, quasiseparated scheme, and let $Z\subseteq X$ be a closed subset with quasicompact complement. Then the category $\D_{qc,Z}(X)$ is weakly approximable and has subcategories
\begin{align*}
\D_{qc,Z}(X)^-&=\D^-_{qc,Z}(X), \D_{qc,Z}(X)^+=\D^+_{qc,Z}(X),\\
\D_{qc,Z}(X)^b&=\D^b_{qc,Z}(X),
\D_{qc,Z}(X)^c=\D_{Z}^\mathsf{per}(X).
\end{align*}
Note that perfect complexes are bounded. Then $\D_{qc,Z}(X)^c\subseteq \D_{qc,Z}(X)^b_c$ and $\D_{qc,Z}(X)^{\rm sb}\subseteq \D_{qc,Z}(X)^b$.
We refer to the big singularity category of $\D_{qc,Z}(X)$ as the
\emph{big singularity category} of $X$, and denote it by
$\D^{\rm big}_{sg,Z}(X)$.

If $X$ is noetherian, then 
$$\D_{qc,Z}(X)^-_c=\D^-_{coh,Z}(X), \D_{qc,Z}(X)^b_c=\D^b_{coh,Z}(X).$$
The small singularity category of $\D_{qc,Z}(X)$ coincides with the
(classical) singularity category of $X$, which is defined by
\[
\D_{sg,Z}(X)\coloneqq
\D^b_{coh,Z}(X)\big/\D^\mathsf{per}_{Z}(X),
\]
following Orlov~\cite{O11}.
\end{exam}

\begin{exam}\label{exam:not algebraic}
Let $\T$ be the homotopy category of spectra. Then $\T$ is an approximable triangulated category (see \cite[Example 3.2]{N18a} or \cite[Remark 5.4]{N18c}).
This example can be generalized to the homotopy categories of module spectra over connective ring spectra
(see \cite[Appendix A]{BCPRZ24}). In general, this kind of triangulated category is \emph{not} necessarily algebraic.

As usual, for a spectrum $M$, we denote by $\pi_n(M)$ the $n$-th homotopy group of $M$ for each $n\in\mathbb{Z}$.
Let $R$ be an $\mathbb{E}_1$-ring spectrum, that is, an $\mathbb{E}_1$-algebra object in the $\infty$-category of spectra.
Suppose that  $R$ is \emph{connective}, namely, $\pi_n(R)=0$ for any $n<0$.
We consider the homotopy category of the stable $\infty$-category of left $R$-module spectra and denote it by $\D(R)$.
It is known that $\D(R)$ is a compactly generated triangulated category with a compact generator $R$.
Since $\Hom_{\D(R)}(R, R[n])\simeq \pi_{-n}(R)=0$ for $n>0$, $\D(R)$ is approximable. Note that each ordinary ring can be regarded as an $\mathbb{E}_1$-ring by taking its Eilenberg-Mac Lane ring spectrum.

We are particularly interested in coherent $\mathbb{E}_1$-ring spectra. Recall that a connective ring spectrum $R$ is said to be \emph{left coherent} if $\pi_0(R)$ is left coherent as an ordinary ring and $\pi_n(R)$ are finitely presented left $\pi_0(R)$-modules for all $n\in\mathbb{Z}$. Now, let $R$ be a left coherent $\mathbb{E}_1$-ring spectrum. Then $\D(R)^c$ is the smallest full triangulated subcategory of $\D(R)$ containing $R$ and closed under direct summands. Moreover, $\D(R)^-_c$ (respectively, $\D(R)^b_c$) is the full subcategory of $\D(R)$ consisting of objects $M$ such that $\pi_n(M)$ are finitely presented $\pi_0(R)$-modules for  $n\in\mathbb{Z}$ and $\pi_n(M)=0$ for $n\ll 0$ (respectively, $|n|\gg 0$) (\cite[Theorem A.5(1), Corollary A.6(1)]{BCPRZ24}).
\end{exam}

\begin{exam}\textnormal{(\cite[Remark 4.4 (3)]{BV20})}
Let $\T$ be a subcategory of $\D(\mathbb{Z})$ compactly generated by $G=\mathbb{Z}/p\mathbb{Z}$ for a prime number $p$. $\T$ is weakly approximable by Proposition~\ref{prop:approximation gene}, but not approximable since for any positive integer $A$, there is an object $F=\mathbb{Z}/p^{A+1}\mathbb{Z}$ that does not satisfy the third condition in the definition of an approximable triangulated category (Definition~\ref{def:appro}(iii)).
\end{exam}

\subsection{Restriction of functors and recollements}\label{sec:restriction}
In this subsection, we study the restriction of functors and recollements in the framework of approximable triangulated categories. Results here are generalizations of the derived categories of finite-dimensional algebras \cite{AKLY17}.

\begin{lem}\label{lem-gluing-t-struc}
{\rm (\cite[Corollary 3.12]{BNP18})}
Let the following diagram be a recollement of weakly approximable triangulated categories
$$\quad\xymatrix{\R\ar^-{i_*=i_!}[r]
&\T\ar^-{j^!=j^*}[r]\ar^-{i^!}@/^1.2pc/[l]\ar_-{i^*}@/_1.6pc/[l]
&\S.\ar^-{j_*}@/^1.2pc/[l]\ar_-{j_!}@/_1.6pc/[l]}$$
Then the glued $t$-structures of t-structure in the preferred equivalence classes are in the preferred equivalence class.
Moreover, this recollement restricts to a left recollement
$$\quad\xymatrix{\R^-\ar^-{i_*=i_!}[r]
&\T^-\ar^-{j^!=j^*}[r]\ar_-{i^*}@/_1.6pc/[l]
&\S^-,\ar_-{j_!}@/_1.6pc/[l]}$$
and a right recollement
$$\quad\xymatrix{\R^+ \ar^-{i_*=i_!}[r]
&{\T^+} \ar^-{j^!=j^*}[r]\ar^-{i^!}@/^1.3pc/[l]
&{{{\S}^+.}}\ar^-{j_*}@/^1.3pc/[l]
}$$
In particular, $i_*$ can be restricted to $\mathcal{R}^b\ra \mathcal{T}^b$ and $j^*$ can be restricted to $\mathcal{T}^b \ra \mathcal{S}^b$. 
\end{lem}

\begin{lem}\label{lem-sub-eq}
{\rm (\cite[Corollary 3.10]{BNP18})}
Let $\mathcal{T}$ be a weakly approximable triangulated category with a compact generator $G$.
Then
\begin{enumerate}
\item  $\mathcal{T}^-=\{X\in \mathcal{T}\mid {\T}(G[i],X)=0 \text{ for }i\ll 0\}$.

\item $\mathcal{T}^+=\{X\in \mathcal{T}\mid {\T}(G[i],X)=0 \text{ for }i\gg 0\}$.

\item $\mathcal{T}^b=\{X\in \mathcal{T}\mid {\T}(G[i],X)=0 \text{ for }|i|\gg 0\}$.
\end{enumerate}
\end{lem}

\begin{lem}\label{lem:functor to -c1}
    Let $\mathbf{F}:\T\to\S$ be a functor between pre-approximable triangulated categories. If $\mathbf{F}$ respects compact objects and coproducts, then $\mathbf{F}$ can be restricted to $\T^-_c\to\S^-_c$.
\end{lem}
\begin{proof}
Let $G_t$ and $G_s$ be compact generators of $\T$ and $\S$, $(\T^{\leq 0}, \T^{\geq 0})$ and $(\S^{\leq 0}, \S^{\geq 0})$ be the $t$-structures generated by $G_t$ and $G_s$, respectively. Since $\mathbf{F}$ respects compact objects, there is a positive integer $n$ such that
$$\mathbf{F}(G_t)\in \langle G_s\rangle_n^{[-n,n]}\subseteq \S^{\leq n},$$
so $\mathbf{F}(G_t[i])\in \S^{\leq n}$ for $i\geq 0$. Note that $\T^{\leq 0}=\overline{\langle G_t\rangle}^{(-\infty, 0]}$ and $\mathbf{F}$ respects coproducts, thus $\mathbf{F}(\T^{\leq 0})\subseteq \overline{\langle \mathbf{F}(G_t)\rangle}^{(-\infty, 0]}\subseteq \S^{\leq n}$.

Let $T\in \T^-_c$.  For any integer $l>0$, let $m=\mathsf{max}\{0, n+l\}$, there is a triangle
$$E\longrightarrow T\longrightarrow D\longrightarrow E[1]$$
with $E\in\T^c$ and $D\in \T^{\leq -m}$. The image of this triangle under $\mathbf{F}$
$$\mathbf{F}(E)\longrightarrow \mathbf{F}(T)\longrightarrow \mathbf{F}(D)\longrightarrow \mathbf{F}(E)[1]$$
satisfies $\mathbf{F}(E)\in\S^c$ and $\mathbf{F}(D)\in \mathbf{F}(\T^{\leq -m})\subseteq \S^{\leq -l}$. Hence $\mathbf{F}(T)\in \S^-_c$.
\end{proof}

\begin{lem}\label{lem:functor to bc}
Let $\T$ and $\S$ be locally Hom-finite approximable $R$-linear triangulated categories.
Let $\mathbf{F}:\T\rightarrow \S$ be a triangulated functor. If $\mathbf{F}$ admits a left adjoint functor that respects compact objects. Then $\mathbf{F}$ can be restricted to $\T^-_c\to \S^-_c$ and $\T^b_c\to \S^b_c$.
\end{lem}
\begin{proof}
 Let $X\in\T^-_c$ and $S\in\S^c$. Assume $\mathbf{F}$ admits a left adjoint $\mathbf{G}$. Then
 $$\S(S[i], \mathbf{F}(X))\simeq \T(\mathbf{G}(S)[i], X).$$
Since $\mathbf{G}(S)\in\T^c$, by Theorem~\ref{lem:object to bc}(1), $\T(\mathbf{G}(S), X)$ is a finite $R$-module and
$$\T(\mathbf{G}(S)[i], X)=0, i\ll 0.$$
Hence, $\S(S, \mathbf{F}(X))$ is a finite $R$-module and $\S(S[i], \mathbf{F}(X))=0, i\ll 0$, and so $\mathbf{F}(X)$ is in $\S^-_c$.

By using Theorem~\ref{lem:object to bc}(2) and similar proofs as above, one can prove $\mathbf{F}$ can be restricted to $\T^b_c\to \S^b_c$.
\end{proof}

\begin{lem}\label{cor:reco to T-c}
Let the following diagram be a recollement of weakly approximable $R$-linear triangulated categories
$$\xymatrix{\R\ar^-{i_*=i_!}[r]
&\T\ar^-{j^!=j^*}[r]\ar^-{i^!}@/^1.2pc/[l]\ar_-{i^*}@/_1.6pc/[l]
&\S.\ar^-{j_*}@/^1.2pc/[l]\ar_-{j_!}@/_1.6pc/[l]}$$
Let $G_{\T}$ be a compact generator of $\T$.
Then the following statements hold.
\begin{enumerate}
\item If $j^*(G_{\T})\in \S_c^-$, then  the upper two rows of the recollement are restricted to a half recollement
$$\xymatrix{\R^-_c\ar^-{i_*=i_!}[r]
&\T^-_c\ar^-{j^!=j^*}[r]\ar_-{i^*}@/_1.6pc/[l]
&\S^-_c.\ar_-{j_!}@/_1.6pc/[l]}$$
In addition, $i_*$ can be restricted to $\R_c^b\to \T_c^b$ and $j^*$ can be restricted to $\T_c^b\to \S_c^b$.
\item If $\R,\S$ and $\T$ are locally Hom-finite approximable $R$-linear triangulated categories, then the upper two rows of the recollement are restricted to a half recollement
$$\xymatrix{\R^-_c\ar^-{i_*=i_!}[r]
&\T^-_c\ar^-{j^!=j^*}[r]\ar_-{i^*}@/_1.6pc/[l]
&\S^-_c.\ar_-{j_!}@/_1.6pc/[l]}$$
Moreover, if $j^*(G_{\T})\in \S^c(\subseteq \S^-_c)$, then the lower two rows of the recollement restrict to a half recollement
$$\xymatrix{\R^b_c\ar^-{i_*=i_!}[r]
&\T^b_c\ar^-{j^!=j^*}[r]\ar^-{i^!}@/^1.2pc/[l]
&\S^b_c.\ar^-{j_*}@/^1.2pc/[l]}$$
\end{enumerate}
\end{lem}
\begin{proof}
(1) By Lemma \ref{lem:functor to -c1}, $i^*$ can be restricted to $\T_c^-\to \R_c^-$ and $j_!$ can be restricted to $\S_c^-\to \T_c^-$.
By \cite[Proposition 4.3]{BNP18}, if $j^*(G_{\T})\in \S_c^-$, then
\begin{align*}
\T_c^-&=\{Z\in \T\mid i^*(Z)\in \R_c^- \text{ and }j^*(Z)\in \S_c^-\}.
\end{align*}
Thus $j^*$ can be restricted to $\T_c^-\to \S_c^-$. For $X\in \R_c^-$, since $j_!$ is fully faithful, we have $i^*(i_*(X))\simeq X\in \R_c^-$. Also $j^*(i_*(X))=0\in \S_c^-$. Then $i_*(X)\in \T_c^-$, and therefore $i_*$ can be restricted to $\R_c^-\to \T_c^-$. This shows that the upper two rows of the recollement are restricted to a half recollement
$$\xymatrix{\R^-_c\ar^-{i_*=i_!}[r]
&\T^-_c\ar^-{j^!=j^*}[r]\ar_-{i^*}@/_1.6pc/[l]
&\S^-_c.\ar_-{j_!}@/_1.6pc/[l]}$$

By \cite[Proposition 4.3]{BNP18}, if $j^*(G_{\T})\in \S_c^-$, then
\begin{align*}
\T_c^b&=\{Z\in \T\mid i^*(Z)\in \R_c^- ,\; j^*(Z)\in \S_c^b\text{ and }i^!(Z)\in \R^+\}.
\end{align*}
Thus $j^*$ can be restricted to $\T_c^b\to \S_c^b$. For $X\in \R_c^b$, since $i_*$ is fully faithful, $i^*i_*(X)\simeq X\in \R_c^b\subseteq \R_c^-$ and $i^!i_*(X)\simeq X\in \R_c^b\subseteq \R^+$. Also $j^*i_*(X)=0\in \S_c^b$. Thus $i_*(X)\in \T_c^b$ and $i_*$ can be restricted to $\R_c^b\to \T_c^b$.

(2) 
Note that $j^*$ admits a left adjoint functor $j_!$ that respects compact objects. By Lemma \ref{lem:functor to bc}, $j^*$ can be restricted to $\T_c^-\to \S_c^-$. In particular, $j^*(G_{\mathcal{T}})\in \S_c^-$. Then the first half recollement is direct from (1). 

By Lemma \ref{lem-rec-ex-down}, it follows from $j^*(G_{\T})\in \S^c$ that there exists a left recollement:
$$\xymatrix{\R^c\ar^-{i_*=i_!}[r]
&\T^c\ar^-{j^!=j^*}[r]
\ar_-{i^*}@/_1.6pc/[l]
&\S^c.\ar_-{j_!}@/_1.6pc/[l]}$$ 
Then the second half recollement follows directly from Lemma~\ref{lem:functor to bc}.
\end{proof}

\section{Localization theorems for $\T^c$ and $\T^b_c$}\label{sec:loc theorem}

In this section, we are to prove localization theorems for approximable triangulated categories. First, let us recall a remarkable theorem due to Neeman.

\begin{lem}\textnormal{(\cite[Theorem 2.1]{N92})}\label{lem:reco to c}
A recollement of compactly generated triangulated categories
    $$\quad\xymatrix{\R\ar^-{i_*=i_!}[r]
&\T\ar^-{j^!=j^*}[r]\ar^-{i^!}@/^1.2pc/[l]\ar_-{i^*}@/_1.6pc/[l]
&\S.\ar^-{j_*}@/^1.2pc/[l]\ar_-{j_!}@/_1.6pc/[l]}$$
restricts to a short exact sequence up to direct summands along the first row
$${\S}^c\stackrel{j_!}\longrightarrow \T^c\stackrel{i^*}\longrightarrow {\R}^c.$$
\end{lem}

\subsection{A localization theorem for $\T^c$}

\begin{thm}\label{thm:exact sequences of quotient}
Let $\R, \S$, and $\T$ be triangulated categories that form a recollement of triangulated categories
$$\quad\xymatrix{\R\ar^-{i_*=i_!}[r]
&\T\ar^-{j^!=j^*}[r]\ar^-{i^!}@/^1.2pc/[l]\ar_-{i^*}@/_1.6pc/[l]
&\S.\ar^-{j_*}@/^1.2pc/[l]\ar_-{j_!}@/_1.6pc/[l]}$$
Then the following statements hold.
\begin{enumerate}
    \item Suppose that $\R, \S$ and $\T$ are pre-approximable. Let $\U,\V$ and $\W$ be triangulated subcategories of $\R, \S$ and $\T$, respectively, such that ${\R}^c\subseteq\U$, ${\S}^c\subseteq\V\subseteq {\S}^-_c$ and $\T^c\subseteq \W$. If the first row of the recollement is restricted to a short exact sequence
$$\V\stackrel{j_!}{\longrightarrow}\W\stackrel{i^*}{\longrightarrow}\U,$$
then it induces a short exact sequence
$$\V/{\S}^c\stackrel{\overline{j_!}}{\longrightarrow}\W/\T^c\stackrel{\overline{i^*}}{\longrightarrow}\U/{\R}^c.$$
  \item Suppose that $\R,\T, \S$ are weakly approximable and that $j^*(G_{\mathcal{T}})\in \mathcal{S}_c^-$, where $G_{\mathcal{T}}$ denotes a compact generator of $\T$.  Then the first row of the recollement induces a short exact sequence $${\S}^-_c/{\S}^c\stackrel{\overline{j_!}}{\longrightarrow}{\T}^-_c/\T^c\stackrel{\overline{i^*}}{\longrightarrow}{\R}^-_c/{\R}^c.$$
In particular, this holds if $\R,\T, \S$ are locally Hom-finite approximable $R$-linear triangulated categories, with $R$ a commutative noetherian ring.  
\end{enumerate}
\end{thm}

\begin{proof}
(1) There is the following commutative diagram
$$\xymatrix{
{\S}^c\ar[r]^{j_!}\ar@{^(->}[d]&\T^c\ar[r]^{i^*}\ar@{^(->}[d]&{\R}^c\ar@{^(->}[d]
\\
\V\ar[r]^{j_!}\ar[d]&\W\ar[r]^{i^*}\ar[d]&\U\ar[d]
\\
\V/\S^c\ar[r]^{\overline{j_!}}&\W/\T^c\ar[r]^{\overline{i^*}}&\U/\R^c
}$$
in which the first row is exact up to direct summands by Theorem \ref{lem:reco to c} and the second row is exact by the hypothesis. It remains to prove that the third row is exact. By Lemma~\ref{lem:third isomorphism}, it suffices to prove that $\overline{j_!}$ is fully faithful. This leads us to check the sufficient conditions in Lemma~\ref{fully-faithful}.

Note that $\overline{j_!}$ has the following decomposition
$$\V/{\S}^c\stackrel{\tilde{j_!}}{\longrightarrow} j_!(\V)/j_!({\S}^c)\stackrel{i}\longrightarrow \W/\T^c.$$
Since $j_!$ is fully faithful, the functor $\tilde{j_!}$ is an equivalence. To prove the functor $i$ is a full embedding, it suffices to check condition (1) in Lemma~\ref{fully-faithful}. It is easy to check $j_!(\V)\cap \T^c=j_!({\S}^c)$. Let $X\in \T^c$ and $Y=j_!(Y')\in j_!(\V)$ and $f$ be any morphism $X\to Y$. Since $X\in \T^c$, we have ${\T}(X, \T^{\leq -n})=0$ for some $n\in\mathbb{N}$. By the proof of Lemma \ref{lem:functor to -c1}, there exists $m\in\mathbb{N}$ such that $j_!(\S^{\leq -m})\subseteq \T^{\leq -n}$. Since $Y'\in\V\subseteq {\S}^-_c$, there is a triangle
$$E\longrightarrow Y'\longrightarrow F\longrightarrow E[1]$$
with $E\in {\S}^c$ and $F\in \S^{\leq -m}$. Applying the functor $j_!$ to the last triangle and using the fact that ${\T}(X, j_!(F))=0$, then we know that $f$ factors through $j_!(E)$. Hence, it follows from Lemma \ref{fully-faithful}(1) that $i$ is a full embedding. Therefore $\overline{j_!}$ is fully faithful.

(2) Due to Lemma~\ref{cor:reco to T-c} and Lemma \ref{lem-rec-es}(1), the first row of the recollement is restricted to a short exact sequence
$${\S}^-_c\stackrel{j_!}{\longrightarrow}{\T}^-_c\stackrel{i^*}{\longrightarrow}{\R}^-_c.$$
Thus there exists a short exact sequence below by (1)
$${\S}^-_c/{\S}^c\stackrel{\overline{j_!}}{\longrightarrow}{\T}^-_c/\T^c\stackrel{\overline{i^*}}{\longrightarrow}{\R}^-_c/{\R}^c.$$

We finish the proof.
\end{proof}

\begin{cor}\label{cor:for Tc}
Let $R$ be a commutative noetherian ring, let $\R, \S$ and $\T$ be locally Hom-finite approximable $R$-linear triangulated categories and admit a recollement of triangulated categories
$$\quad\xymatrix{\R\ar^-{i_*=i_!}[r]
&\T\ar^-{j^!=j^*}[r]\ar^-{i^!}@/^1.2pc/[l]\ar_-{i^*}@/_1.6pc/[l]
&\S.\ar^-{j_*}@/^1.2pc/[l]\ar_-{j_!}@/_1.6pc/[l]}$$
Suppose that $\T$ has a compact generator $G_{\T}$ such that there is an integer $N$ with $\T(G_{\T}, G_{\T}[n])=0, n<N$, and $j^*(G_{\T})\in \S^c$. 
Then the second row induces a commutative diagram of short exact sequences
$$\xymatrix{
\R^b_c/\R^c\ar[r]^{\overline{i_*}}\ar@{^(->}[d]&\T^b_c/\T^c\ar[r]^{\overline{j^*}}\ar@{^(->}[d]&\S^b_c/\S^c\ar@{^(->}[d]
\\
\R^-_c/\R^c\ar[r]^{\overline{i_*}}&\T^-_c/\T^c\ar[r]^{\overline{j^*}}&\S^-_c/\S^c.
}$$
\end{cor}
\begin{proof}
Since the four functors in the top two rows respect compact objects, $i^*(G_{\T})$ and $j^*(G_{\T})$ are compact generators of $\R$ and $\S$, respectively, and satisfy
$$\T(i^*(G_{\T}), i^*(G_{\T})[\ll 0])\simeq \T(G_{\T}, i_*i^*(G_{\T})[\ll 0])=0$$
and
$$\T(j^*(G_{\T}), j^*(G_{\T})[\ll 0])\simeq \T(j_!j^*(G_{\T}), G_{\T}[\ll 0])=0.$$
It follows from Lemma~\ref{lem:Tc in Tbc} that $\R^b_c, \T^b_c$ and $\S^b_c$ contain
$\R^c, \T^c$ and $\S^c$, respectively. By Lemma~\ref{cor:reco to T-c}(2) and Lemma \ref{lem-rec-es}(1), the second row is restricted to a short exact sequence
$$\R^b_c\stackrel{i_*}{\longrightarrow}\T^b_c\stackrel{j^*}{\longrightarrow}\S^b_c.$$
Hence, the commutative diagram follows from Theorem~\ref{thm:exact sequences of quotient}.
\end{proof}

Compared with \cite[Theorem 6.1]{JYZ23}, the triangulated categories there should be algebraic, i.e. the derived categories of DG categories, but there are examples of approximable triangulated categories which are not algebraic (see Example~\ref{exam:not algebraic}). The proof of \cite[Theorem 6.1]{JYZ23} heavily depends on the properties of the derived categories of DG categories.

The commutative diagram in the corollary above induces a short sequence of Verdier quotient categories
$$\R^-_c/\R^b_c\stackrel{\overline{i_*}}\longrightarrow\T^-_c/\T^b_c\stackrel{\overline{j^*}}\longrightarrow\S^-_c/\S^b_c,$$
it is proved to be exact in the following subsection with an additive assumption of $\R, \T$ and $\S$ being all coherent.

\subsection{A localization theorem for $\T^b_c$}
Let $\T$ be a pre-approximable triangulated category. We say that $\T$ is {\em coherent} (\cite[Definition 5.1]{N18c}) if there exists an integer $N>0$ such that for every object $X\in \T^-_c$ there exists a triangle
$$A\longrightarrow X\longrightarrow B\longrightarrow A[1]$$
with $A\in \T^-_c\cap \T^{\leq 0}$ and $B\in\T^-_c\cap \T^{\geq -N}=\T_c^b\cap \T^{\geq -N}$. This notion is a slight relaxation of the assertion that a $t$-structure in the preferred equivalence class can be restricted to a $t$-structure on $\T^-_c$. The notion of coherence plays a crucial role in the localization
theorem for $\T_c^b$ proved below.

Let $\R, \S$ and $\T$ be triangulated categories. Suppose that there is a recollement
$$(\star)\quad
\xymatrix{\R\ar^-{i_*=i_!}[r]
&\T\ar^-{j^!=j^*}[r]\ar^-{i^!}@/^1.2pc/[l]\ar_-{i^*}@/_1.6pc/[l]
&\S.\ar^-{j_*}@/^1.2pc/[l]\ar_-{j_!}@/_1.6pc/[l]}$$
Throughout this subsection we assume that
$\R,\S$ and $\T$ are weakly approximable and coherent,
and that $j^*(G_{\T})\in \S_c^-$,
where $G_{\T}$ is a compact generator of $\T$.

The key technical ingredient is the following lemma,
which controls morphisms from $i_*(\R_c^-)$ to objects in $\T^+$.
\begin{lem}\label{lem:intersection}
In the recollement $(\star)$, there are
\begin{enumerate}
    \item $\T_c^b\cap i_*(\R_c^-)=i_*(\R^b_c).$
    \item For any $f:X\to Y$ with $X\in i_*(\R_c^-)$ and $Y\in \T^{+}$, $f$ factors through some object in $i_*(\R_c^b)$.
\end{enumerate}

\end{lem}
\begin{proof}
Let $f: X\to Y$ be a morphism in $\T$ with $X=i_*(X_1)$ for some $X_1\in {\R}_{c}^{-}$ and $Y\in \T^+$. Due to $Y\in \T^{+}$, there exists $n\in\mathbb{N}$ such that $\T(\T^{\leq -n}, Y)=0$. Since $i_*$ is $t$-exact with respect to $t$-structures and their gluing, and the gluing $t$-structure of $t$-structures in the preferred equivalence class belongs to the preferred equivalence class (Lemma~\ref{lem-gluing-t-struc}), then there is an integer $m>0$ such that $i_{*}(\R^{\leq -m})\subseteq \T^{\leq -n}$. Because $\R$ is coherent, then for $X_1$, there is a triangle
$$U_1\stackrel{\alpha}\longrightarrow X_1\stackrel{\beta}\longrightarrow V_1\longrightarrow U_1[1]$$
with $U_1\in \R^{\leq -m}\cap {\R}_{c}^{-}$ and $V_1\in {\R}_{c}^b$. By applying the functor $i_*$ to the triangle above, we have
$$\xymatrix{
i_*(U_1)\ar[r]^-{i_*(\alpha)}&X\ar[d]^f\ar[r]^-{i_*(\beta)}&i_*(V_1)\ar@{-->}[dl]\ar[r]&i_*(U_1)[1]
\\
&Y&&
}$$
since $\T(i_*(U_1), Y)=0$, $f$ factors through $i_*(V_1)$ which belongs to $i_{*}({\R}_{c}^b)$. This completes the proof of (2).

Applying (2) to the identity morphism $f=\mathrm{id}_X$,
we obtain that $i_*(\alpha)=0$. Note that $i_*$ is fully faithful. It follows that $\alpha =0$ and $X_1$ is a summand of $V_1\in {\R}_{c}^b$. Then $X_1\in {\R}_{c}^b$ and $X= i_{*}(X_1)\in i_{*}({\R}_{c}^b)$. Thus $\T_c^b\cap i_*(\R_c^-)\subseteq i_*(\R^b_c)$. By Lemma \ref{cor:reco to T-c}(1), $i_*(\R^b_c) \subseteq \T_c^b$, and therefore $i_*(\R^b_c) \subseteq \T_c^b\cap i_*(\R_c^-)$. Hence, $i_*(\R^b_c) = \T_c^b\cap i_*(\R_c^-)$. This completes the proof of (1).
\end{proof}

Using the above lemma we prove that the second row of the recollement
restricts to $\T_c^b$.
\begin{thm}\label{thm:exact seq Tbc}
 Given the recollement $(\star)$, the following statements hold.
 \begin{enumerate}
  \item The second row in the recollement is restricted to a short exact sequence up to direct summands
$$(\#)\quad\R^b_c\stackrel{i_*}\longrightarrow \T_c^b\stackrel{j^*}\longrightarrow \S^b_c.$$
\item Let $\U,\V$ and $\W$ be triangulated subcategories of $\R, \S$ and $\T$, respectively, such that ${\R}^b_c\subseteq\U\subseteq {\R}^-_c$, ${\S}^b_c\subseteq\V$ and $\T^b_c\subseteq \W$. If the second row of the recollement is restricted to a short exact sequence
$$\U\stackrel{i_*}{\longrightarrow}\W\stackrel{j^*}{\longrightarrow}\V,$$
then it induces a short exact sequence
$$\U/{\R}^b_c\stackrel{\overline{i_*}}{\longrightarrow}\W/\T^b_c\stackrel{\overline{j^*}}{\longrightarrow}\V/{\S}^b_c.$$
\item The second row induces a short exact sequence
$${\R}^-_c/{\R}^b_c\stackrel{\overline{i_*}}{\longrightarrow}\T^-_c/\T^b_c\stackrel{\overline{j^*}}{\longrightarrow}{\S}^-_c/{\S}^b_c.$$
 \end{enumerate}
\end{thm}
\begin{proof}
It follows directly from Lemma~\ref{cor:reco to T-c}(1) that $i_*$ and $j^*$ in $(\#)$ are well-defined and $i_*$ is fully faithful, and from Lemma~\ref{lem:intersection}(1) that $\T_c^b\cap i_{*}(\R^-_c)=i_{*}(\R^b_c)$. Then $j^*$ induces a triangle functor $\overline{j^*}:{\T^b_c}/{i_*({\R}^b_c)}\to {\S}^b_c$ with $\overline{j^*}(T)={j^*}(T)$ and $\overline{j^*}(f/s)=j^*(f){j^*(s)}^{-1}$. The discussion is displayed in a diagram
$$\xymatrix{
{{\R}^b_c}\ar[r]^{i_*}&{{\T}^b_c}\ar[r]^{\textnormal{cano.}\quad\quad}\ar[d]_{j^*}&{{\T}^b_c}/{i_*({\R}^b_c)}\ar@{-->}[ld]^{\overline{j^*}}
\\
&\S^b_c
}$$

Claim 1: $\overline{j^*}$ is fully faithful.

There is the following commutative diagram
$$\xymatrix{
i_{*}({\R}_{c}^b)\ar@{^(->}[r]\ar@{^(->}[d]&
\T_{c}^{b}\ar^{\textnormal{cano.}\quad\quad}[r]\ar@{^(->}[d]&
\T_{c}^{b}/i_{*}({\R}_{c}^b)\ar_{\mathbf{J}}[d]\\
i_{*}({\R}_{c}^{-})\ar@{^(->}[r]&
\T_{c}^{-}\ar^{\textnormal{cano.}\quad\quad}[r]&
\T_{{c}}^{-}/i_{*}({\R}_{c}^{-}).
}$$
By Lemma~\ref{lem:intersection}(2) and condition (2) in Lemma~\ref{fully-faithful}, the functor  $\mathbf{J}$ is fully faithful.

By Lemma~\ref{cor:reco to T-c}(1), $j^*$ induces a fully faithful functor $\widetilde{j^*}:\T^-_c/ i_*({\R}_c^-)\to {\S}_{c}^{-}$. Then the composition
$$\T_{c}^{b}/i_{*}({\R}_{c}^b)\stackrel{\mathbf{J}}\longrightarrow \T_{c}^{-}/i_{*}({\R}_{c}^{-})\stackrel{\widetilde{j^*}}{\longrightarrow} {\S}_{c}^{-}$$
is a fully faithful functor. Note that there is the following commutative diagram
$$\xymatrix{
\T_{c}^{b}/i_{*}({\R}_{c}^b)\ar[r]^{\quad\quad\overline{j^{*}}}\ar[d]_{\mathbf{J}}&{\S}_{c}^b\ar@{^(->}[d]
\\
\T_{c}^{-}/i_{*}({\R}_{c}^{-})\ar[r]^{\quad\quad\widetilde{j^*}}&{\S}_{c}^{-}
}$$
Hence $\overline{j^{*}}$ is fully faithful.

Claim 2: $\overline{j^{*}}$ is dense up to direct summands.

Let $X\in {\S}_{c}^b$, there exists $n\in\mathbb{N}$ such that $\S(\S^{\leq -n}, X)=0$ and $j_!(X)\in \T_{c}^{-}$.
Because $j^*$ is right $t$-exact with respect to $t$-structures and their gluing, and the gluing $t$-structure of $t$-structures in the preferred equivalence class belongs to the preferred equivalence class  (Lemma~\ref{lem-gluing-t-struc}), then there exists an integer $m>0$ such that $j^*(\T^{\leq -m})\subseteq \S^{\leq -n}$.
Since $\T$ is coherent, then we have a triangle
$$U\stackrel{f}\longrightarrow j_!(X)\longrightarrow V\longrightarrow U[1]$$
with $U\in \T^{\leq -m}\cap \T_{c}^{-}$ and $V\in \T_{c}^{b}$. By applying the functor $j^*$ to the triangle above, there is a triangle
$$j^*(U)\stackrel{j^*(f)}\longrightarrow j^*j_!(X)\longrightarrow j^*(V)\longrightarrow j^*(U)[1].$$
Note that $j^*(U)\in j^*(\T^{\leq -m})\subseteq \S^{\leq -n}$, and so $\S(j^*(U), j^*j_!(X))=0$. This implies that
 $j^{*}(f)=0$, thus $j^*j_!(X)\simeq X$ is a direct summand of $j^{*}(V)$ which belongs to $j^*(\T_c^b)$.

 (2) It is from (1) that there is a commutative diagram
 $$\xymatrix{
{\R}^b_c\ar[r]^{i_*}\ar@{^(->}[d]&\T^b_c\ar[r]^{j^*}\ar@{^(->}[d]&{\S}^b_c\ar@{^(->}[d]
\\
\U\ar[r]^{i_*}\ar[d]&\W\ar[r]^{j^*}\ar[d]&\V\ar[d]
\\
\U/\R^b_c\ar[r]^{\overline{i_*}}&\W/\T^b_c\ar[r]^{\overline{j^*}}&\V/\S^b_c
}$$
in which the first row is exact up to direct summands and the second row is exact, we want to show the third one is exact. Thanks to Lemma~\ref{lem:third isomorphism}, it suffices to prove $\overline{i_*}$ is fully faithful, and this leads us to use Lemma \ref{fully-faithful}(2). It follows from Lemma~\ref{lem:intersection}(1) that $i_*(\U)\cap \T_c^b=i_*(\R^b_c)$. Note that there is the following commutative diagram
$$\xymatrix{
i_*({\R}^b_c)\ar[r]\ar@{^(->}[d]&i_*(\U)\ar[r]\ar@{^(->}[d]&i_*(\U)/i_*({\R^b_c})\ar[d]
\\
\T^b_c\ar[r]&\W\ar[r]&\W/{\T}^b_c.
}$$
Let $f: i_*(Y)\to Z$ be a morphism in $\W$ with $Y\in \U$ and $Z\in \T_c^b$, by Lemma \ref{lem:intersection}(2), it factors through some object in $i_*(\R^b_c)$. Hence, by Lemma~\ref{lem:third isomorphism} the following sequence
$$\U/{\R}^b_c\stackrel{\overline{i_*}}{\longrightarrow}\W/\T^b_c\stackrel{\overline{j^*}}{\longrightarrow}\V/{\S}^b_c$$
is exact.

(3) Thanks to Lemma~\ref{cor:reco to T-c}(1), the second row of the recollement restricts to a short exact sequence
$$\R_c^-\stackrel{i_*}{\longrightarrow}\T_c^-\stackrel{j^*}{\longrightarrow}\S_c^-.$$
Thus there is a short exact sequence
$$\R_c^-/\R_c^b\stackrel{i_*}{\longrightarrow}\T_c^-/\T_c^b\stackrel{j^*}{\longrightarrow}\S_c^-/\S_c^b.$$
\end{proof}

Combining the previous localization results we obtain
the following $3\times 3$ diagram of short exact sequences.
\begin{cor}\label{cor:for Tbc}
With the notation and assumptions in Corollary~\ref{cor:for Tc}. If  moreover, $\R, \T, \S$ are coherent, then the diagram there can be completed into a commutative diagram
$$\xymatrix{
\R^b_c/\R^c\ar[r]^{\overline{i_*}}\ar@{^(->}[d]&\T^b_c/\T^c\ar[r]^{\overline{j^*}}\ar@{^(->}[d]&\S^b_c/\S^c\ar@{^(->}[d]
\\
\R^-_c/\R^c\ar[r]^{\overline{i_*}}\ar[d]&\T^-_c/\T^c\ar[r]^{\overline{j^*}}\ar[d]&\S^-_c/\S^c\ar[d]
\\
\R^-_c/\R^b_c\ar[r]^{\overline{i_*}}&\T^-_c/\T^b_c\ar[r]^{\overline{j^*}}&\S^-_c/\S^b_c}$$
in which all rows and columns are short exact sequences.
\end{cor}
\begin{proof}
In this case, the second row can be restricted to short exact sequences
 $$\R^b_c\stackrel{i_*}{\longrightarrow}\T^b_c\stackrel{j^*}{\longrightarrow}\S^b_c$$
 and
 $$\R_c^-\stackrel{i_*}{\longrightarrow}\T_c^-\stackrel{j^*}{\longrightarrow}\S_c^-.$$
By Theorem \ref{thm:exact sequences of quotient} and Theorem \ref{thm:exact seq Tbc}, the three rows are exact. The exactness of the three columns and the commutativity of the diagram are obvious.
\end{proof}


\subsection{A localization theorem for $\mathcal{T}^{{\rm sb}}$}
In this subsection we prove a localization theorem for the
subcategory $\mathcal{T}^{\rm sb}$ associated with a recollement.
The key step is the following factorization lemma.
\begin{lem}\label{lem-l-sp}
Let the following diagram be a recollement of weakly approximable triangulated categories
$$\xymatrix{\R\ar^-{i_*=i_!}[r]
&\T\ar^-{j^!=j^*}[r]\ar^-{i^!}@/^1.2pc/[l]\ar_-{i^*}@/_1.6pc/[l]
&\S.\ar^-{j_*}@/^1.2pc/[l]\ar_-{j_!}@/_1.6pc/[l]}$$
Then
\begin{enumerate}
\item  $j_!$ can be restricted to $\mathcal{S}^{{\rm sb}}\ra \mathcal{T}^{{\rm sb}}$ and $i^*$ can be restricted to $\mathcal{T}^{{\rm sb}}\ra \mathcal{R}^{{\rm sb}}$. 

\item  Every morphism from an object in $\mathcal{T}^{{\rm sb}}$ to an object in $j_!(\mathcal{S}^{-})$ factors through an object in $j_!(\mathcal{S}^{{\rm sb}})$. In particular, if $\mathcal{S}^{{\rm sb}}\subseteq\mathcal{V}\subseteq\mathcal{S}^{-}$, then $j_!(\mathcal{V})\cap\mathcal{T}^{{\rm sb}}=j_!(\mathcal{S}^{{\rm sb}})$.
\end{enumerate}
\end{lem}
\begin{proof}
(1) We first show that $j_!$ restricts to
$\mathcal{S}^{\rm sb}\to\mathcal{T}^{\rm sb}$. Let $Y\in \mathcal{S}^{{\rm sb}}$.
For $Z'\in \mathcal{T}^-$, $j^*(Z')\in \mathcal{S}^-$ by Lemma \ref{lem-gluing-t-struc}. 
Then ${\mathcal{S}}(Y,j^*(Z')[i])=0$ for $i\gg0$ by Lemma \ref{lem-sp=big}.
Then
${\mathcal{T}}(j_!(Y),Z'[i])\simeq {\mathcal{S}}(Y,j^*(Z')[i])=0$ for $i\gg 0$.
By Lemma \ref{lem-sp=big}, we have $j_!(Y)\in \mathcal{T}^{{\rm sb}}$.
Similarly, $i^*$ can be restricted to $\mathcal{S}^{{\rm sb}}\ra \mathcal{R}^{{\rm sb}}$. 

(2) By Lemma \ref{lem-gluing-t-struc}, we can assume that $(\mathcal{T}^{\le 0},\mathcal{T}^{\ge 0})$ is the glued t-structure. Let $X\in\mathcal{T}^{{\rm sb}}$ and $Y\in j_{!}(\mathcal{S}^{-})$. Then there exist some $k \in \mathbb{Z}$ and $Y' \in \mathcal{S}^{\leq k}$ such that $j_!(Y') \simeq Y$. By Definition \ref{def:appro}(ii),  an easy induction, we obtain that for every $m > 0$, there is a canonical triangle\[E_{m}\stackrel{\alpha_{m}}{\lra} Y'\stackrel{\beta_{m}}{\lra} D_{m}\lra E_{m}[1],\]with $E_{m}\in\mathcal{S}^{{\rm sb}}$ and $D_{m}\in\mathcal{S}^{\leq k-m}$. Applying the functor $j_!$ yields a triangle
\[j_!(E_{m}) \stackrel{j_!(\alpha_{m})}{\lra} j_!(Y')\simeq Y\stackrel{j_!(\beta_{m})}{\lra} j_!(D_{m})\lra j_!(E_{m})[1],\]where $j_{!}(E_{m})\in\mathcal{T}^{{\rm sb}}$ by (1) and $j_{!}(D_{m})\in\mathcal{T}^{\leq k-m}$ by the definition of the glued t-structure. Lemma \ref{lem-sp=big} implies that ${\mathcal{T}}(X, j_!(D_m)) = 0$ for $m \gg 0$. Therefore, every morphism in ${\mathcal{T}}(X, Y)$ factors through $j_!(E_m)$ for sufficiently large $m$. 

If $X\simeq Y$, then $j_!(\beta_m)=0$. Since $j_!$ is fully faithful, $\beta_m=0$ and $Y'$ is a summand of $E_m$. Then $Y'\in \mathcal{S}^{\rm sb}$ and $j_!*(Y')\in j_!( \mathcal{S}^{\rm sb})$. Thus $j_!(\mathcal{S}^-)\cap\mathcal{T}^{{\rm sb}}\subseteq j_!(\mathcal{S}^{{\rm sb}})$. Note that $j_!(\mathcal{S}^{{\rm sb}})\subseteq j_!(\mathcal{V})\cap\mathcal{T}^{{\rm sb}}\subseteq j_!(\mathcal{S}^-)\cap\mathcal{T}^{{\rm sb}}$. Hence $j_!(\mathcal{S}^{{\rm sb}})= j_!(\mathcal{V})\cap\mathcal{T}^{{\rm sb}}$.
This completes the proof.
\end{proof}

\begin{thm}\label{thm-l-sp}
Let  $\R, \S$ and $\T$ be weakly approximable triangulated categories admitting a recollement of triangulated categories
$$\quad\xymatrix{\R\ar^-{i_*=i_!}[r]
&\T\ar^-{j^!=j^*}[r]\ar^-{i^!}@/^1.2pc/[l]\ar_-{i^*}@/_1.6pc/[l]
&\S.\ar^-{j_*}@/^1.2pc/[l]\ar_-{j_!}@/_1.6pc/[l]}$$
Then the following statements hold.
\begin{enumerate}
\item The first row of recollement is restricted to a short exact sequence up to direct summands
\begin{align*}
\xymatrixcolsep{2pc}\xymatrix{
{\mathcal{S}}^{{\rm sb}} \ar[r]^{j_!} &\mathcal{T}^{{\rm sb}} \ar[r]^{i^*}  &{\mathcal{R}}^{{\rm sb}}.
}
\end{align*}

\item Let $\mathcal{U}\subseteq \R,\mathcal{V}\subseteq \S$ and $\mathcal{W}\subseteq \T$ be triangulated subcategories satisfying ${\mathcal{R}}^{{\rm sb}} \subseteq\mathcal{U}$, ${\mathcal{S}}^{{\rm sb}}\subseteq\mathcal{V}\subseteq \mathcal{S}^-$ and $\mathcal{T}^{{\rm sb}}\subseteq \mathcal{W}$. 
If the first row of the recollement is restricted to a short exact sequence
\begin{align*}
\xymatrixcolsep{2pc}\xymatrix{
\mathcal{V} \ar[r]^{j_!} &\mathcal{W} \ar[r]^{i^*}  &\mathcal{U},
}
\end{align*}
then it induces a short exact sequence
\begin{align*}
\xymatrixcolsep{2pc}\xymatrix{
\mathcal{V}/\mathcal{S}^{{\rm sb}}\ar[r]^{j_!} &\mathcal{W}/\mathcal{T}^{{\rm sb}} \ar[r]^{i^*}  &\mathcal{U}/\mathcal{R}^{{\rm sb}}.
}
\end{align*}

\item The first row induces a short exact sequence
\begin{align*}
\xymatrixcolsep{2pc}\xymatrix{
\mathcal{S}^-/\mathcal{S}^{{\rm sb}}\ar[r]^{j_!} &\mathcal{T}^-/\mathcal{T}^{{\rm sb}} \ar[r]^{i^*}  &\mathcal{R}^-/\mathcal{R}^{{\rm sb}}.
}
\end{align*}
\end{enumerate}
\end{thm}

\begin{proof}
(1) By Lemma \ref{lem-l-sp}(1), $j_!$ can be restricted to $\mathcal{S}^{{\rm sb}}\ra \mathcal{T}^{{\rm sb}}$ and $i^*$ can be restricted to $\mathcal{T}^{{\rm sb}}\ra \mathcal{R}^{{\rm sb}}$. 
Note that $j_!$ is fully faithful and $i^*j_!=0$.
Then $i^*$ induces a triangle functor $\overline{i^*}:\mathcal{T}^{{\rm sb}}/j_!(\mathcal{S}^{{\rm sb}})\ra \mathcal{R}^{{\rm sb}}$ such that the following diagram is commutative.
\begin{align*}
\xymatrixcolsep{4pc}\xymatrix{
\mathcal{S}^{{\rm sb}} \ar[r]^{j_!} &\mathcal{T}^{{\rm sb}} \ar[d]^{i^*} \ar[r]^{\text{can.}}& \mathcal{T}^{{\rm sb}}/j_!(\mathcal{S}^{{\rm sb}})\ar@{-->}[ld]^{\overline{i^*}} \\
&\mathcal{R}^{{\rm sb}}&
}
\end{align*}

Claim 1: $\overline{i^*}$ is fully faithful.

By Lemma \ref{lem-l-sp}(1), $j_{!}(\mathcal{S}^{{\rm sb}})\subseteq\mathcal{T}^{{\rm sb}}$ .
Note that  $\mathcal{S}^{{\rm sb}}\subseteq\mathcal{S}^{-}$ and $\mathcal{T}^{{\rm sb}}\subseteq\mathcal{T}^{-}$.
Then we have the following commutative diagram:
$$\xymatrix{
j_!(\mathcal{S}^{{\rm sb}})\ar@{^(->}[r] \ar@{^(->}[d]&\mathcal{T}^{{\rm sb}}\ar[r]\ar@{^(->}[d]&\mathcal{T}^{{\rm sb}}/j_!(\mathcal{S}^{{\rm sb}})\ar[d]^{\mathbf{I}}\\
j_!(\mathcal{S}^-)\ar@{^(->}[r]&\mathcal{T}^-\ar[r]&\mathcal{T}^-/j_!(\mathcal{S}^-).
}$$
Combining Lemma \ref{lem-l-sp}(2) with Lemma \ref{fully-faithful}, we conclude that the functor $\mathbf{I}$ is fully faithful.
By Lemma \ref{lem-gluing-t-struc} and Lemma \ref{lem-rec-es}, 
the first row in the recollement is restricted to a short exact sequence
$\xymatrix{
{\mathcal{S}}^-\ar[r]^{j_!} &\mathcal{T}^- \ar[r]^{i^*}  &{\mathcal{R}}^-.
}$
Then $i^*$ induces a fully faithful functor $\widetilde{i^*}:\mathcal{T}^-/j_!(\mathcal{S}^-)\ra \mathcal{R}^-$. 
Hence the composition
$$\xymatrix{
\mathcal{T}^{{\rm sb}}/j_!(\mathcal{S}^{{\rm sb}}) \ar[r]^{\mathbf{I}}&\mathcal{T}^-/j_!(\mathcal{S}^-)\ar[r]^-{\widetilde{i^*}}&\mathcal{R}^-
}$$
is also fully faithful. Now consider the commutative diagram:
$$\xymatrix{
\mathcal{T}^{{\rm sb}}/j_!(\mathcal{S}^{{\rm sb}})\ar[d]^{\mathbf{I}}\ar[r]^-{\overline{i^*}}&\mathcal{R}^{{\rm sb}}\ar@{^(->}[d]\\
\mathcal{T}^-/j_!(\mathcal{S}^-)\ar[r]^-{\widetilde{i^*}}&\mathcal{R}^-.
}$$
It follows that $\overline{i^*}$ is fully faithful.

Claim 2:  $\overline{i^*}$ is dense up to direct summands.

By Lemma \ref{lem-gluing-t-struc}, without loss of generality, we can assume that $(\mathcal{T}^{\leq 0},\mathcal{T}^{\ge 0})$ is the glued t-structure associated to the given $t$-structures on $\mathcal{R}$ and $\mathcal{S}$. For $X\in \mathcal{R}^{{\rm sb}}$, it follows from $\mathcal{R}^{{\rm sb}}\subseteq \mathcal{R}^-$ that $X \in \mathcal{R}^{\leq k}$ for some $k \in \mathbb{Z}$. By the definition of the glued t-structure, $i_{*}(X)\in \mathcal{T}^{\leq k}$. By Defition \ref{def:appro}(ii),  an easy induction, we obtain that for each $m>0$, there is a canonical triangle
$$E_{m} \stackrel{\alpha_{m}}{\lra} i_*(X)\stackrel{\beta_{m}}{\lra} D_{m}\lra E_{m}[1]$$
with $E_{m}\in \mathcal{T}^{{\rm sb}}$ and $D_{m}\in \mathcal{T}^{\le k-m}$.
Applying the functor $i^*$ to the above triangle, we get a triangle
$$i^*(E_{m}) \stackrel{i^*(\alpha_{m})}{\lra} i^*(i_*(X))\stackrel{i^*(\beta_{m})}{\lra} i^*(D_{m})\lra i^{*}(E_{m})[1],$$
where $i^*(E_{m})\in i^{*}(\mathcal{T}^{{\rm sb}})$ and $i^*(D_{m})\in \mathcal{R}^{\le k-m}$ by the definition of the glued t-structure. Note that $X\simeq i^*(i_*(X))$. By Lemma \ref{lem-sp=big}, we have $i^*(\beta_m) = 0$ for $m \gg 0$.
Therefore, $X$ is a direct summand of $i^*(E_m)$ for $m \gg 0$, which completes the proof.

(2) By (1), we have a commutative diagram
$$\xymatrix{
{\mathcal{S}}^{{\rm sb}} \ar[r]^{j_!} \ar@{^(->}[d]&\mathcal{T}^{{\rm sb}} \ar[r]^{i^*}  \ar@{^(->}[d]&{\mathcal{R}}^{{\rm sb}} \ar@{^(->}[d]\\
\mathcal{V} \ar[r]^{j_!} \ar[d]&\mathcal{W} \ar[r]^{i^*}  \ar[d]&\mathcal{U} \ar[d]\\
\mathcal{V}/{\mathcal{S}}^{{\rm sb}}  \ar[r]^{\overline{j_!}} &\mathcal{W}/\mathcal{T}^{{\rm sb}} \ar[r]^{\overline{i^*}} &\mathcal{U}/{\mathcal{R}}^{{\rm sb}},
}$$
where the first row is exact up to direct summands and the second row is exact.
We have to show that the third one is exact. 
By Lemma \ref{lem:third isomorphism}, it suffices to prove that $\overline{j_!}$ is fully faithful.
Thanks to Lemma \ref{lem-l-sp}(2), $j_!(\mathcal{V})\cap\mathcal{T}^{{\rm sb}}=j_!(\mathcal{S}^{{\rm sb}})$.
Then by Lemma \ref{fully-faithful}, we obtain a commutative diagram
$$\xymatrix{
j_!({\mathcal{S}}^{{\rm sb}})\ar@{^(->}[r] \ar@{^(->}[d]
&j_!(\mathcal{V})\ar[r]\ar@{^(->}[d]
&j_!(\mathcal{V})/j_!({\mathcal{S}}^{{\rm sb}})\ar[d]^{\mathbf{J}}\\
\mathcal{T}^{{\rm sb}}\ar@{^(->}[r]
&\mathcal{W}\ar[r]
&\mathcal{W}/\mathcal{T}^{{\rm sb}}
}$$
with $\mathbf{J}$ fully faithful. 
Thus $\overline{j_!}$ is fully faithful.
Hence, by Lemma \ref{lem:third isomorphism}, we get a short exact sequence
$$\xymatrix{
\mathcal{V}/{\mathcal{S}}^{{\rm sb}}  \ar[r]^{\overline{j_!}} &\mathcal{W}/\mathcal{T}^{{\rm sb}} \ar[r]^{\overline{i^*}} &\mathcal{U}/{\mathcal{R}}^{{\rm sb}}.
}$$

(3) By Lemma \ref{lem-gluing-t-struc} and Lemma \ref{lem-rec-es}, the first row induces a short exact sequence
\begin{align*}
\xymatrixcolsep{2pc}\xymatrix{
\mathcal{S}^-\ar[r]^{j_!} &\mathcal{T}^-\ar[r]^{i^*}  &\mathcal{R}^-.
}
\end{align*}
Then we get the desired short exact sequences by (2). 
\end{proof}

\subsection{A localization theorem for $\mathcal{T}^b$}

In this subsection, we establish an analogous localization theorem
for the bounded subcategory $\mathcal{T}^b$.
The following lemma controls morphisms involving the image of $i_*$.
\begin{lem}\label{lem-l-sb}
Let the following diagram be a recollement of weakly approximable triangulated categories
$$\xymatrix{\R\ar^-{i_*=i_!}[r]
&\T\ar^-{j^!=j^*}[r]\ar^-{i^!}@/^1.2pc/[l]\ar_-{i^*}@/_1.6pc/[l]
&\S.\ar^-{j_*}@/^1.2pc/[l]\ar_-{j_!}@/_1.6pc/[l]}$$
Then 
\begin{enumerate}
\item every morphism from an object in $i_*(\mathcal{R}^-)$ to an object in $\mathcal{T}^+$ factors through an object in $i_*(\mathcal{R}^b)$. In particular, if $\mathcal{R}^{b}\subseteq\mathcal{U}\subseteq\mathcal{R}^{-}$, then $i_*(\mathcal{U})\cap \mathcal{T}^b=i_*(\mathcal{R}^b)=i_*(\mathcal{U})\cap \mathcal{T}^{+}$.

\item every morphism from an object in $\mathcal{T}^-$ to an object in $i_*(\mathcal{R}^+)$ factors through an object in $i_*(\mathcal{R}^b)$. In particular, if $\mathcal{R}^{b}\subseteq\mathcal{U}\subseteq\mathcal{R}^{+}$, then $i_*(\mathcal{U})\cap \mathcal{T}^b=i_*(\mathcal{R}^b)=i_*(\mathcal{U})\cap \mathcal{T}^{-}$.
\end{enumerate}
\end{lem}
\begin{proof}

(1) By Lemma \ref{lem-gluing-t-struc}, without loss of generality, we can assume that $(\mathcal{T}^{\leq 0},\mathcal{T}^{\ge 0})$ is the glued t-structure. Let $X\in i_{*}(\mathcal{R}^{-})$ and $Y\in\mathcal{T}^{+}$. Then there exists $X'\in\mathcal{R}^{\leq r}\subseteq \mathcal{R}^-$ for some $r\in\mathbb{Z}$ such that $i_{*}(X')\simeq X$ and there exists an integer $k\in\mathbb{Z}$ such that $Y\in\mathcal{T}^{\geq k}\subseteq \mathcal{T}^+$. By the definition of the glued t-structure, we have $X\simeq i_{*}(X')\in i_{*}(\mathcal{R}^{\leq r})\subseteq \mathcal{T}^{\leq r}$. If $r<k$, then ${\mathcal{T}}(X,Y)=0$, and (1) holds trivially. Now suppose that $r\geq k$. Note that there is a canonical triangle \[X'^{\leq k-1}\xrightarrow{f}X'\to X'^{\geq k}\to X'^{\leq k-1}[1]\]with $X'^{\leq k-1}\in\mathcal{R}^{\leq k-1}$ and $X'^{\geq k}\in\mathcal{R}^{b}$. Applying the functor $i_{*}$ yields a triangle \[i_{*}(X'^{\leq k-1})\xrightarrow{i_{*}(f)}i_{*}(X')\to i_{*}(X'^{\geq k})\to i_{*}(X'^{\leq k-1})[1],\]where $i_{*}(X'^{\leq k-1})\in\mathcal{T}^{\leq k-1}$ by the definition of the glued t-structure and $i_{*}(X'^{\geq k})\in\mathcal{T}^{b}$ by Lemma \ref{lem-gluing-t-struc}. Since ${\mathcal{T}}(i_{*}(X'^{\leq k-1}),Y)=0$, any morphism in ${\mathcal{T}}(i_*(X'), Y)$ must factor through $i_*(X'^{\geq k})$. In particular, if $X\simeq Y$, then $i_*(f)=0$. Since $i_*$ is fully faithful, $f=0$ and $X'$
 is a summand of ${X'}^{\ge k}\in \mathcal{R}^b$.
Then $X'\in \mathcal{R}^b$ and $i_*(X')\in i_*(\mathcal{R}^b)$. Thus $i_*(\mathcal{U})\cap \mathcal{T}^{+}\subseteq i_*(\mathcal{R}^b)$. Clearly, $i_*(\mathcal{R}^b)\subseteq i_*(\mathcal{U})\cap \mathcal{T}^{b}\subseteq i_*(\mathcal{U})\cap \mathcal{T}^{+}$. Hence, $i_*(\mathcal{U})\cap \mathcal{T}^{+}= i_*(\mathcal{R}^b)=i_*(\mathcal{U})\cap \mathcal{T}^{b}$. 

(2) is dual to (1).
\end{proof}

\begin{thm}\label{thm-l-sb}
Let  $\R, \S$ and $\T$ be weakly approximable triangulated categories admitting a recollement of triangulated categories
$$\quad\xymatrix{\R\ar^-{i_*=i_!}[r]
&\T\ar^-{j^!=j^*}[r]\ar^-{i^!}@/^1.2pc/[l]\ar_-{i^*}@/_1.6pc/[l]
&\S.\ar^-{j_*}@/^1.2pc/[l]\ar_-{j_!}@/_1.6pc/[l]}$$
Then the following statements hold.
\begin{enumerate}
\item The second row in the recollement is restricted to a short exact sequence up to direct summands
\begin{align*}
\xymatrixcolsep{2pc}\xymatrix{\mathcal{R}^b \ar[r]^{i_*} &\mathcal{T}^b \ar[r]^{j^*}  &\mathcal{S}^b. 
}
\end{align*}

\item Let $\mathcal{U}\subseteq \R,\mathcal{V}\subseteq \S$ and $\mathcal{W}\subseteq \T$ be triangulated subcategories satisfying ${\mathcal{R}}^b \subseteq\mathcal{U} \subseteq {\mathcal{R}}^-$ (or ${\mathcal{R}}^b \subseteq\mathcal{U} \subseteq {\mathcal{R}}^+$), ${\mathcal{S}}^b\subseteq\mathcal{V}$ and $\mathcal{T}^b\subseteq \mathcal{W}$. 
If the second row of the recollement is restricted to a short exact sequence
\begin{align*}
\xymatrixcolsep{2pc}\xymatrix{
\mathcal{U} \ar[r]^{i_*} &\mathcal{W} \ar[r]^{j^*}  &\mathcal{V},
}
\end{align*}
then it induces a short exact sequence
\begin{align*}
\xymatrixcolsep{2pc}\xymatrix{
\mathcal{U}/\mathcal{R}^b  \ar[r]^{i_*} &\mathcal{W}/\mathcal{T}^b \ar[r]^{j^*}  &\mathcal{V}/\mathcal{S}^b. 
}
\end{align*}

\item  The second row induces short exact sequences
\begin{align*}
\xymatrix{
\mathcal{R}^-/\mathcal{R}^b  \ar[r]^{i_*} &\mathcal{T}^-/\mathcal{T}^b \ar[r]^{j^*}  &\mathcal{S}^-/\mathcal{S}^b,\\
\mathcal{R}^+/\mathcal{R}^b  \ar[r]^{i_*} &\mathcal{T}^+/\mathcal{T}^b \ar[r]^{j^*}  &\mathcal{S}^+/\mathcal{S}^b. 
}
\end{align*}
\end{enumerate}
\end{thm}
\begin{proof}
(1) By Lemma \ref{lem-gluing-t-struc}, $i_*$ can be restricted to $\mathcal{R}^b\ra \mathcal{T}^b$ and $j^*$ can be restricted to $\mathcal{T}^b \ra \mathcal{S}^b$. 
Note that $i_*$ is fully faithful and $j^*i_*=0$. 
Then $j^*: \mathcal{T}^b\ra \mathcal{S}^b$ induces a triangle functor $\overline{j^*}:\mathcal{T}^b/i_*(\mathcal{R}^b)\ra \mathcal{S}^b$ such that the following diagram is commutative.

Claim 1: $\overline{j^*}$ is fully faithful.

Consider the commutative diagram
$$\xymatrix{
i_*(\mathcal{R}^b)\ar@{^(->}[r] \ar@{^(->}[d]&\mathcal{T}^b\ar[r]\ar@{^(->}[d]&\mathcal{T}^b/i_*(\mathcal{R}^b)\ar[d]^{\mathbf{J}}\\
i_*(\mathcal{R}^-)\ar@{^(->}[r]&\mathcal{T}^-\ar[r]&\mathcal{T}^-/i_*(\mathcal{R}^-).
}$$
By combining Lemma \ref{lem-l-sb}(1) with Lemma \ref{fully-faithful}, we see that the functor $\mathbf{J}$ is fully faithful.
By Lemma \ref{lem-gluing-t-struc} and Lemma \ref{lem-rec-es}, 
the second row in the recollement is restricted to a short exact sequence
$\xymatrix{
{\mathcal{R}}^-\ar[r]^{i_*} &\mathcal{T}^- \ar[r]^{j^*}  &{\mathcal{S}}^-.
}$
Then $j^*$ induces a fully faithful functor $\widetilde{j^*}:\mathcal{T}^-/i_*(\mathcal{R}^-)\ra \mathcal{S}^-$. 
Hence the composition
$$\xymatrix{
\mathcal{T}^b/i_*(\mathcal{R}^b) \ar[r]^{\mathbf{J}}&\mathcal{T}^-/i_*(\mathcal{R}^-)\ar[r]^-{\widetilde{j^*}}&\mathcal{S}^-
}$$
is a fully faithful functor.
Note that there is the following commutative diagram
$$\xymatrix{
\mathcal{T}^b/i_*(\mathcal{R}^b)\ar[r]^-{\overline{j^*}}\ar[d]^{\mathbf{J}}&\mathcal{S}^b\ar@{^(->}[d]\\
\mathcal{T}^-/i_*(\mathcal{R}^-)\ar[r]^-{\widetilde{j^*}}&\mathcal{S}^-.
}$$
Thus $\overline{j^*}$ is fully faithful.

Claim 2:  $\overline{j^{*}}$ is dense up to direct summands.

By Lemma \ref{lem-gluing-t-struc}, without loss of generality, we can assume that $(\mathcal{T}^{\leq 0},\mathcal{T}^{\geq 0})$ is the glued t-structure. Let $X\in\mathcal{S}^{b}$. Then there exist $r\leq k$ such that $X\in\mathcal{S}^{\leq k}\cap\mathcal{S}^{\geq r}$. By the definition of the glued t-structure, $j_{!}(X)\in j_{!}(\mathcal{S}^{\leq k})\subseteq \mathcal{T}^{\leq k}$. Then there is a canonical triangle
\[(j_{!}(X))^{\leq r-1}\xrightarrow{f}j_{!}(X)\to(j_{!}(X))^{\geq r}\to(j_{!}(X))^{\leq r-1}[1]\]such that $(j_{!}(X))^{\leq r-1}\in\mathcal{T}^{\leq r-1}$ and $(j_{!}(X))^{\geq r}\in\mathcal{T}^{\geq r}$. 
By Lemma \ref{lem-gluing-t-struc}, $j_!(\mathcal{S}^-)\subseteq \mathcal{T}^-$.
Then $j_!(X)\in \mathcal{T}^-$.
It follows from $(j_{!}(X))^{\leq r-1}\in\mathcal{T}^{\leq r-1}\subseteq \mathcal{T}^-$ that $(j_{!}(X))^{\geq r}\in \mathcal{T}^-$. Thus $(j_{!}(X))^{\geq r}\in\mathcal{T}^b$. 

Applying the functor $j^{*}$ yields a triangle
\[j^{*}((j_{!}(X))^{\leq r-1})\xrightarrow{j^{*}(f)}j^{*}j_{!}(X)\to j^{*}((j_{!}(X))^{\geq r})\to j^{*}((j_{!}(X))^{\leq r-1})[1],\]
where $j^{*}((j_{!}(X))^{\leq r-1})\in j^*(\mathcal{T}^{\leq r-1})\subseteq \mathcal{S}^{\leq r-1}$ by the definition of the glued t-structure. Since $j^{*}j_{!}(X)\simeq X\in\mathcal{S}^{\geq r}$, we have $j^{*}(f)=0$. Therefore, $X$ is a direct summand of $j^{*}(j_{!}(X))^{\geq r})\in j^{*}(\mathcal{T}^b)$, which completes the proof.

(2) By (1), we have a commutative diagram
$$\xymatrix{
\mathcal{R}^b\ar[r]^{i_*}\ar@{^(->}[d]&\mathcal{T}^b\ar[r]^{j^*}\ar@{^(->}[d]&\mathcal{S}^b\ar@{^(->}[d]\\
\mathcal{U}\ar[r]^{i_*}\ar[d]&\mathcal{W}\ar[r]^{j^*}\ar[d]&\mathcal{V}\ar[d]\\
\mathcal{U}/\mathcal{R}^b\ar[r]^{\overline{i_*}}&\mathcal{W}/\mathcal{T}^b\ar[r]^{\overline{j^*}}&\mathcal{V}/\mathcal{S}^b,
}$$
where the first row is exact up to direct summands and the second row is exact.
We have to show that the third one is exact. 
By Lemma \ref{lem:third isomorphism}, it suffices to prove that $\overline{i_*}$ is fully faithful. Consider the commutative diagram
$$\xymatrix{
i_*(\mathcal{R}^b)\ar@{^(->}[r] \ar@{^(->}[d]&i_*(\mathcal{U})\ar[r]\ar@{^(->}[d]&i_*(\mathcal{U})/i_*(\mathcal{R}^b)\ar[d]^{\mathbf{J}}\\
\mathcal{T}^b\ar@{^(->}[r]&\mathcal{W}\ar[r]&\mathcal{W}/\mathcal{T}^b.
}$$Thanks to Lemma \ref{lem-l-sb}, $i_*(\mathcal{U})\cap\mathcal{T}^b=i_*(\mathcal{R}^b)$.
By Lemma \ref{lem-l-sb}, if $\mathcal{U} \subseteq {\mathcal{R}}^-$, then every morphism from an object in $i_*(\mathcal{U})$ to an object in $\mathcal{T}^b$ factors through an object in $i_*(\mathcal{R}^b)$.
Then $\mathbf{J}$ is fully faithful by Lemma \ref{fully-faithful}(1).
(By Lemma \ref{lem-l-sb}, if $\mathcal{U} \subseteq {\mathcal{R}}^+$, then every morphism from an object in $\mathcal{T}^b$ to an object in $i_*(\mathcal{U})$ factors through an object in $i_*(\mathcal{R}^b)$.
Then $\mathbf{J}$ is fully faithful by Lemma \ref{lem:third isomorphism}(2).)
Thus $\overline{i_*}$ is fully faithful.
Hence, by Lemma \ref{lem:third isomorphism}, we get a short exact sequence
\begin{align*}
\xymatrixcolsep{2pc}\xymatrix{
\mathcal{U}/\mathcal{R}^b  \ar[r]^{i_*} &\mathcal{W}/\mathcal{T}^b \ar[r]^{j^*}  &\mathcal{V}/\mathcal{S}^b. 
}
\end{align*}

(3) By Lemma \ref{lem-gluing-t-struc}, the second row induces short exact sequences
\begin{align*}
\xymatrixcolsep{2pc}\xymatrix{\mathcal{R}^- \ar[r]^{i_*} &\mathcal{T}^-\ar[r]^{j^*}  &\mathcal{S}^-,
}\\
\xymatrixcolsep{2pc}\xymatrix{\mathcal{R}^+ \ar[r]^{i_*} &\mathcal{T}^+\ar[r]^{j^*}  &\mathcal{S}^+. 
}
\end{align*}
Then we get the desired short exact sequences by (2). 
\end{proof}

\begin{prop}\label{prop-rec-res}
\textnormal{(\cite{ccz25})}
Let the following diagram be a recollement of weakly approximable triangulated categories
$$\xymatrix{\R\ar^-{i_*=i_!}[r]
&\T\ar^-{j^!=j^*}[r]\ar^-{i^!}@/^1.2pc/[l]\ar_-{i^*}@/_1.6pc/[l]
&\S.\ar^-{j_*}@/^1.2pc/[l]\ar_-{j_!}@/_1.6pc/[l]}$$
Let $G_{\mathcal{R}}$ and $G_{\mathcal{T}}$ be compact generators of $\mathcal{R}$ and $\mathcal{T}$, respectively.
Then the following are equivalent.
\begin{enumerate}
\item[(1) ] $i_{*}(G_{\mathcal{R}})\in \mathcal{T}^{{\rm sb}}$.
\item[($1'$)]  $j^{*}(G_{\mathcal{T}})\in \mathcal{S}^{{\rm sb}}$.

\item[(2) ] The recollement induces a left recollement
$$\xymatrix{\R^{{\rm sb}}\ar^-{i_*=i_!}[r]
&\T^{{\rm sb}}\ar^-{j^!=j^*}[r]\ar_-{i^*}@/_1.6pc/[l]
&\S^{{\rm sb}}.\ar_-{j_!}@/_1.6pc/[l]}$$

\item[(3) ] The recollement induces a recollement
$$\xymatrix{\R^- \ar^-{i_*=i_!}[r]
&\T^- \ar^-{j^!=j^*}[r]\ar^-{i^!}@/^1.2pc/[l]\ar_-{i^*}@/_1.6pc/[l]
&\S^- .\ar^-{j_*}@/^1.2pc/[l]\ar_-{j_!}@/_1.6pc/[l]}$$
\end{enumerate}

Moreover, each of the above conditions implies:
\begin{enumerate}
\item[(4) ]  The recollement induces a right recollement
$$\xymatrix{\R^b\ar^-{i_*=i_!}[r]
&\T^b\ar^-{j^!=j^*}[r]
\ar^-{i^!}@/^1.3pc/[l]
&\S^b.\ar^-{j_*}@/^1.3pc/[l]}$$
\end{enumerate}

If, in addition, ${\mathcal{R}}(G_{\mathcal{R}},G_{\mathcal{R}}[i])={\mathcal{T}}(G_{\mathcal{T}},G_{\mathcal{T}}[i])$ for $i\ll 0$, then {\rm (4)} is also equivalent to {\rm (1)–(3)}.
\end{prop}

\begin{lem}\label{subfinal1}
Let the following diagram be a recollement of weakly approximable triangulated categories
$$\xymatrix{\R\ar^-{i_*=i_!}[r]
&\T\ar^-{j^!=j^*}[r]\ar^-{i^!}@/^1.2pc/[l]\ar_-{i^*}@/_1.6pc/[l]
&\S\ar^-{j_*}@/^1.2pc/[l]\ar_-{j_!}@/_1.6pc/[l]}$$
and let $G_{\mathcal{T}}\in\mathcal{T}$ be a compact generator. Suppose that $j^{*}(G_{\mathcal{T}})\in\mathcal{S}^{{\rm sb}}$. Then the following statements hold.

\begin{enumerate}
\item Every morphism from an object in $\mathcal{T}^{{\rm sb}}$ to an object in $i_{*}(\mathcal{R}^{-})$ factors through an object in $i_{*}(\mathcal{R}^{{\rm sb}})$. In particular, if $\mathcal{R}^{{\rm sb}}\subseteq\mathcal{U}\subseteq\mathcal{R}^{-}$, then $i_{*}(\mathcal{U})\cap\mathcal{T}^{{\rm sb}}=i_{*}(\mathcal{R}^{{\rm sb}})$.

\item Let $\mathcal{U},\mathcal{V}$ and $\mathcal{W}$ be triangulated subcategories of $\mathcal{R}, \mathcal{T}$ and $\mathcal{S}$, respectively. 
Suppose ${\mathcal{R}}^{{\rm sb}} \subseteq\mathcal{U} \subseteq {\mathcal{R}}^-$, $\mathcal{S}^{{\rm sb}}\subseteq \mathcal{V}$ and  ${\mathcal{T}}^{{\rm sb}}\subseteq\mathcal{W}$. 
If the second row of the recollement is restricted to a short exact sequence
\begin{align*}
\xymatrixcolsep{2pc}\xymatrix{
\mathcal{U} \ar[r]^{i_*} &\mathcal{W} \ar[r]^{j^*}  &\mathcal{V},
}
\end{align*}
then the second row of the recollement induces a short exact sequence
\begin{align*}
\xymatrixcolsep{2pc}\xymatrix{
\mathcal{U}/\mathcal{R}^{{\rm sb}}  \ar[r]^{i_*} &\mathcal{W}/\mathcal{T}^{{\rm sb}} \ar[r]^{j^*}  &\mathcal{V}/\mathcal{S}^{{\rm sb}}. 
}
\end{align*}
\end{enumerate}
\end{lem}
\begin{proof}
By Proposition \ref{prop-rec-res} and Lemma \ref{lem-rec-es},
the second row of the recollement induces a short exact sequence
$$\xymatrix{\R^{{\rm sb}}\ar^-{i_*}[r]
&\T^{{\rm sb}}\ar^-{j^!}[r]
&\S^{{\rm sb}}.}$$
The desired conclusions then follow from arguments similar to those used in the proofs of Lemma~\ref{lem-l-sp} and Theorem~\ref{thm-l-sp}.
\end{proof}

\begin{cor}\label{cor-3-3-pb-}
Let the following diagram be a recollement of weakly approximable triangulated categories
$$\xymatrix{\R\ar^-{i_*=i_!}[r]
&\T\ar^-{j^!=j^*}[r]\ar^-{i^!}@/^1.2pc/[l]\ar_-{i^*}@/_1.6pc/[l]
&\S.\ar^-{j_*}@/^1.2pc/[l]\ar_-{j_!}@/_1.6pc/[l]}$$
Suppose that $\T$ has a compact generator $G_{\T}$ such that there is an integer $N$ with $\T(G_{\T}, G_{\T}[n])=0$ for $n<N$, and $j^*(G_{\T})\in \S^{\rm sb}$. Then the recollement induces a commutative diagram
$$\xymatrix{
\mathcal{R}^b/\mathcal{R}^{{\rm sb}}\ar[r]^{\overline{i_*}}\ar@{^(->}[d]&\mathcal{T}^b/\mathcal{T}^{{\rm sb}}\ar[r]^{\overline{j^*}}\ar@{^(->}[d]&\mathcal{S}^b/\mathcal{S}^{{\rm sb}}\ar@{^(->}[d]\\
\mathcal{R}^-/\mathcal{R}^{{\rm sb}}\ar[r]^{\overline{i_*}}\ar[d]&\mathcal{T}^-/\mathcal{T}^{{\rm sb}}\ar[r]^{\overline{j^*}}\ar[d]&\mathcal{S}^-/\mathcal{S}^{{\rm sb}}\ar[d]\\
\mathcal{R}^-/\mathcal{R}^b\ar[r]^{\overline{i_*}}&\mathcal{T}^-/\mathcal{T}^b\ar[r]^{\overline{j^*}}&\mathcal{S}^-/\mathcal{S}^b
}$$
in which all rows and columns are short exact sequences.
\end{cor}
\begin{proof}
Let $G_{\mathcal{R}}$ and $G_{\mathcal{S}}$ be a compact generator of $\mathcal{R}$ and $\mathcal{S}$, respectively. Assume that ${\mathcal{T}}(G_{\mathcal{T}},G_{\mathcal{T}}[i])=0$ for all $i\le n$ for some $n\in \mathbb{Z}$. We claim that ${\mathcal{S}}(G_{\mathcal{S}},G_{\mathcal{S}}[i])=0$ for $i\ll 0$.
Indeed, since $j_!$ can be restricted to $\mathcal{S}^c\ra \mathcal{T}^c$, we have $j_!(G_{\mathcal{S}})\in \mathcal{T}^c$.
Then there are $u\le v\in \mathbb{Z}$ such that $j_!(G_{\mathcal{S}})\in \langle G_{\mathcal{T}}\rangle^{[u,v]}$.
Thus ${\mathcal{T}}(Y_1,Y_2[i])=0$ for $Y_1,Y_2\in \langle G_{\mathcal{T}}\rangle^{[u,v]}$ and $i\le n+u-v$.
Since $j_!$ is fully faithful, we have ${\mathcal{S}}(G_{\mathcal{S}},G_{\mathcal{S}}[i])={\mathcal{T}}(j_!(G_{\mathcal{S}}),j_!(G_{\mathcal{S}})[i])=0$ for all $i\le n+u-v$.

Next, we claim that ${\mathcal{R}}(G_{\mathcal{R}},G_{\mathcal{R}}[i])=0$ for $i\ll 0$. 
Indeed, due to $i_{*}(G_{\mathcal{R}})\in \mathcal{T}^{{\rm sb}}$, there exists $u\ge 0$ such that $i_*(G_{\mathcal{R}})\in \overline{\langle G_{\mathcal{T}}\rangle}^{[-u,u]}$. Since ${\mathcal{T}}(G_{\mathcal{T}},G_{\mathcal{T}}[i])=0$ for all $i\le n$, we have $\overline{\langle G_{\mathcal{T}}\rangle}^{[-u,u]}[i]\subseteq (G_{\mathcal{T}})^{\perp}$ when $i\leq n-u$. Hence $\overline{\langle G_{\mathcal{T}}\rangle}^{[-u,u]}[j]\subseteq G_{\mathcal{T}}[-u,u]^{\perp}$ when $j\leq n-2u$. Fix $Y\in\overline{\langle G_{\mathcal{T}}\rangle}^{[-u,u]}[j]$ for some $j\leq n-2u$. Then $G_{\mathcal{T}}[-u,u]\subseteq {}^{\perp}Y$. Note that ${}^{\perp}Y$ is closed under coproducts and extensions, we conclude that $\overline{\langle G_{\mathcal{T}}\rangle}^{[-u,u]}\subseteq{}^{\perp}Y$. This means, for any $X\in\overline{\langle G_{\mathcal{T}}\rangle}^{[-u,u]}$ and $Y\in\overline{\langle G_{\mathcal{T}}\rangle}^{[-u,u]}[j]$ for some $j\le n-2u$, ${\mathcal{T}}(X,Y)=0$. Since $i_*$ is fully faithful, ${\mathcal{R}}(G_{\mathcal{R}},G_{\mathcal{R}}[i])\simeq {\mathcal{T}}(i_{*}(G_{\mathcal{R}}),i_{*}(G_{\mathcal{R}})[i])=0$ for all $i\le n-2u$.

It follows from Lemma \ref{lem-sp=big} that $\mathcal{R}^b$, $\mathcal{T}^b$ and $\mathcal{S}^b$ contain
$\mathcal{R}^{{\rm sb}}$, $\mathcal{T}^{{\rm sb}}$ and $\mathcal{S}^{{\rm sb}}$, respectively.
By Lemma \ref{lem-gluing-t-struc} and Lemma \ref{lem-rec-es}, the recollement induces short exact sequences
$$\xymatrix{
\mathcal{R}^- \ar[r]^-{i_*} &\mathcal{T}^-\ar[r]^-{j^*}  &\mathcal{S}^-\text{ and } \
\mathcal{R}^b \ar[r]^-{i_*} &\mathcal{T}^b\ar[r]^-{j^*}  &\mathcal{S}^b.
}
$$
Then this corollary follows from Lemma \ref{subfinal1} and Theorem \ref{thm-l-sb}.
\end{proof}

\section{Applications}\label{sec:applications}
In this section we apply the localization theorems established in
Section~\ref{sec:loc theorem} to several important classes of
triangulated categories, including derived categories of algebras, homologically smooth DG algebras, and schemes.

We first consider recollements of derived categories of algebras.
\subsection{Localization sequences in the derived categories of algebras}
Let $R$ be a commutative noetherian ring. An $R$-algebra $A$ is called 
{\em finite} if it is finitely generated as an $R$-module. Suppose that 
$A$ is a finite $R$-algebra. Then $A$ is noetherian, and $\D(A)$ is a 
locally Hom-finite, coherent, approximable $R$-linear triangulated 
category (\cite[Remark 5.4]{N18c}).

\begin{cor}\label{cor:finite-dim alg}
Let the following diagram be a recollement of the derived categories of  $R$-algebras $A, B$ and $C$
$$\xymatrix{\D(B)\ar^-{i_*=i_!}[r]
&\D(A)\ar^-{j^!=j^*}[r]\ar^-{i^!}@/^1.3pc/[l]\ar_-{i^*}@/_1.3pc/[l]
&\D(C).\ar^-{j_*}@/^1.3pc/[l]\ar_-{j_!}@/_1.3pc/[l]}$$
Then the following statements hold. 
\begin{enumerate}
\item The first row of the recollement restricts to a short exact sequence up to direct summands
$$\K^b(C\Proj)\stackrel{j_!}\longrightarrow \K^b(A\Proj)\stackrel{i^*}\longrightarrow \K^b(B\Proj).$$
The second row in the recollement is restricted to a short exact sequence up to direct summands
$$\D^b(B\Mod)\stackrel{i_*}\longrightarrow \D^b(A\Mod)\stackrel{j^*}\longrightarrow \D^b(C\Mod).$$
If, in addition, $i_*(B)\in \K^b(A\Proj)$, then  the second row induces the following commutative diagram
{\xiaowuhao
$$\xymatrix{
\D^b(B\Mod)/\K^b(B\Proj)\ar[r]^{\overline{i_*}}\ar@{^(->}[d]
&\D^b(A\Mod)/\K^b(A\Proj)\ar[r]^{\overline{j^*}}\ar@{^(->}[d]
&\D^b(C\Mod)/\K^b(C\Proj)\ar@{^(->}[d]
\\
\D^-(B\Mod)/\K^b(B\Proj)\ar[r]^{\overline{i_*}}\ar[d]
&\D^-(A\Mod)/\K^b(A\Proj)\ar[r]^{\overline{j^*}}\ar[d]
&\D^-(C\Mod)/\K^b(C\Proj)\ar[d]
\\
\D^-(B\Mod)/\D^b(B\Mod)\ar[r]^{\overline{i_*}}&\D^-(A\Mod)/\D^b(A\Mod)\ar[r]^{\overline{j^*}}&\D^-(C\Mod)/\D^b(C\Mod)
}$$}
in which all rows and columns are short exact sequences.

\item
Assume further that $A$, $B$, and $C$ are finite.
Then the second row in the recollement is restricted to a short exact sequence up to direct summands
$$\D^b(B\smod)\stackrel{i_*}\longrightarrow \D^b(A\smod)\stackrel{j^*}\longrightarrow \D^b(C\smod).$$
If, moreover, $j^*(A)$ is isomorphic in $\D(C\Mod)$ to a bounded complex of finitely generated projective $C$-modules,
then the second row induces the following commutative diagram
{\xiaowuhao
$$\xymatrix{
\D_{\mathrm{sg}}(B)\ar[r]^{\overline{i_*}}\ar@{^(->}[d]&\D_{\mathrm{sg}}(A)\ar[r]^{\overline{j^*}}\ar@{^(->}[d]&\D_{\mathrm{sg}}(C)\ar@{^(->}[d]
\\
\D^-(B\smod)/\K^b(B\proj)\ar[r]^{\overline{i_*}}\ar[d]&\D^-(A\smod)/\K^b(A\proj)\ar[r]^{\overline{j^*}}\ar[d]&\D^-(C\smod)/\K^b(C\proj)\ar[d]
\\
\D^-(B\smod)/\D^b(B\smod)\ar[r]^{\overline{i_*}}&\D^-(A\smod)/\D^b(A\smod)\ar[r]^{\overline{j^*}}&\D^-(C\smod)/\D^b(C\smod)
}$$}
in which all rows and columns are short exact sequences.
\end{enumerate}
\end{cor}
\begin{proof}
(1) By Theorem \ref{thm-l-sp}(1), the sequence
$$\K^b(C\Proj)\stackrel{j_!}\longrightarrow \K^b(A\Proj)\stackrel{i^*}\longrightarrow \K^b(B\Proj)$$
is exact up to direct summands. 
 By Theorem \ref{thm-l-sb}(1), the sequence
$$\D^b(B\Mod)\stackrel{i_*}\longrightarrow \D^b(A\Mod)\stackrel{j^*}\longrightarrow \D^b(C\Mod)$$
is exact up to direct summands.
The claimed commutative diagram then follows directly from Corollary \ref{cor-3-3-pb-}.

(2) By Theorem \ref{thm:exact seq Tbc}(1), the sequence
$$\D^b(B\smod)\stackrel{i_*}\longrightarrow \D^b(A\smod)\stackrel{j^*}\longrightarrow \D^b(C\smod)$$
is exact up to direct summands. The corresponding commutative diagram follows directly from Corollary \ref{cor:for Tbc}.
\end{proof}

The short exact sequence of the singularity categories 
$$\D_{\mathrm{sg}}(B)\stackrel{\overline{i_*}}\longrightarrow \D_{\mathrm{sg}}(A)\stackrel{\overline{j^*}}\longrightarrow \D_{\mathrm{sg}}(C)$$ is also obtained by Jin-Yang-Zhou \cite[Theorem 1.1]{JYZ23} differently by using DG theory. The other short exact sequences seem certainly mathematically correct and interesting, but we have not found any relevant discussions.

\subsection{Localization sequences in the derived categories of DG algebras}\label{subsec:DG algebra}
Let $k$ be a field and let $A$ be a DG $k$-algebra satisfying
\begin{enumerate}
    \item $A$ is homologically smooth, that is $\D_{\mathrm{fd}}(A) \subseteq \operatorname{per}(A)$;
    \item $A^i=0$ for any $i > 0$;
    \item $\mathsf{H}^0(A)$ is finite-dimensional as a $k$-space.
\end{enumerate}
Let $\T=\D(A)$, the condition (2) implies that $\T$ is approximable. Due to \cite[Proposition 2.5]{KY16}, $\mathsf{H}^i(A)$ is finite-dimensional for each $i\in\mathbb{Z}$, then $\T$ is a locally Hom-finite approximable $k$-linear triangulated category. Thanks to Theorem~\ref{lem:object to bc}, we obtain $\T^b_c=\D_{fd}(A)$, where
$$\D_{fd}(A):=\{X\in \D(A)\mid\mathsf{H}^{i}(X)~\text{is finite-dimensional and}~\mathsf{H}^{i}(X)=0~\text{for}~|i|\gg 0\}.$$
The AGK category {(\cite{A09, G10})} of $A$ is defined as
$$\D_{\mathsf{agk}}(A):=\operatorname{per}(A)/\D_{fd}(A)=\T^c/\T^b_c.$$

\begin{prop}\label{prop:DG algebra}
Let $A$ be a DG $k$-algebra satisfying (1)$\sim$(3) above. Then $\D(A)$ is a locally Hom-finite coherent approximable $k$-linear triangulated category.
\end{prop}
\begin{proof}
 It is known that $\D(A)$ is a locally Hom-finite approximable $k$-linear triangulated category. We only need to check that it is coherent. Note that
 the standard $t$-structure can be restricted to $\D_{fd}^{-}(A)$, where
 $$\D_{fd}^{-}(A):=\{X\in \D(A)\mid\mathsf{H}^{i}(X)~\text{is finite-dimensional and}~\mathsf{H}^{i}(X)=0~\text{for}~i\gg 0\}.$$
So it suffices to prove that $\D(A)_c^{-}=\D_{fd}^{-}(A)$.
The inclusion $\D(A)_c^{-}\subseteq D_{fd}^{-}(A)$ is direct from their definitions and the characterization (Theorem~\ref{lem:object to bc}(1)) of $\D(A)_c^{-}$.

Let $X\in \D_{fd}^{-}(A)$, for any $m\in\mathbb{N}$, there exists a triangle for $X$
$$X_{1}\longrightarrow X\longrightarrow X_{2}\longrightarrow X_1[1]$$
with $X_{1}\in \D(A)^{\leq -m}$ and $X_2\in \D(A)^{\geq -m+1}$. Since $X\in \D_{fd}^{-}(A)$, then $X_2\in \D_{fd}(A)\subseteq {\rm per}(A)$. For $X_2$ there is the canonical triangle with respect to the co-$t$-structure $({\langle A\rangle}^{[0,+\infty)}, {\langle A\rangle}^{(-\infty,0]})$ in per$(A)$ (see \cite[Proposition 2.3]{KY16})
$$U\longrightarrow X_2\longrightarrow V\longrightarrow U[1]$$
with $U\in {\langle A\rangle}^{[-m,+\infty)}$ and $V\in {\langle A\rangle}^{(-\infty,-m-1]}$. Note that $\Hom_{\D(A)}(U,X_1[1])=0$. Then there is the following octahedral diagram
$$\xymatrix{
&&U\ar[d]\ar@{-->}[dl]&
\\
X_1\ar[r]\ar@{-->}[d]&X\ar[r]\ar@{-->}[dl]&X_2\ar[r]\ar[d]&X_1[1]
\\
V'\ar@{-->}[rr]&&V&
}$$
this produces a triangle
$$U\longrightarrow X\longrightarrow V'\longrightarrow U[1]$$
with $U\in$ per$(A)$ and $V'\in \D(A)^{\leq -m}$. Thus $X\in \D(A)_c^{-}$. We finish the proof.
\end{proof}

\begin{cor}
Let $A, B$ and $C$ be DG $k$-algebras satisfying (1)$\sim$(3) above. If there is a recollement of the derived categories
$$\xymatrix{\D(B)\ar^-{i_*=i_!}[r]
&\D(A)\ar^-{j^!=j^*}[r]\ar^-{i^!}@/^1.3pc/[l]\ar_-{i^*}@/_1.3pc/[l]
&\D(C).\ar^-{j_*}@/^1.3pc/[l]\ar_-{j_!}@/_1.3pc/[l]}$$
Then the second row in the recollement is restricted to a short exact sequence up to direct summands
$$\D_{fd}(B)\stackrel{i_*}\longrightarrow \D_{fd}(A)\stackrel{j^*}\longrightarrow \D_{fd}(C).$$
Moreover,  if $j^*(A)\in \D^c(C\Mod)$, then the second row induces the following commutative diagram
{\xiaowuhao
$$\xymatrix{
\D_{\mathrm{agk}}(B)\ar[r]^{\overline{i_*}}\ar[d]&\D_{\mathrm{agk}}(A)\ar[r]^{\overline{j^*}}\ar[d]&\D_{\mathrm{agk}}(C)\ar[d]
\\
\D^-_{fd}(B)/\D_{fd}(B)\ar[r]^{\overline{i_*}}\ar[d]&\D^-_{fd}(A)/\D_{fd}(A)\ar[r]^{\overline{j^*}}\ar[d]&\D^-_{fd}(C)/\D_{fd}(C)\ar[d]
\\
\D^-_{fd}(B)/\per(B)\ar[r]^{\overline{i_*}}&\D^-_{fd}(A)/\per(A)\ar[r]^{\overline{j^*}}&\D^-_{fd}(C)/\per(C)
}$$}
in which all rows and columns are short exact sequences.
\end{cor}
\begin{proof}
From Proposition~\ref{prop:DG algebra} and its proof, $\D(B), \D(A)$ and $\D(C)$ are locally Hom-finite coherent approximable $k$-linear triangulated categories, $\D(?)^c=\per(?)$, $\D(?)^b_c=\D_{fd}(?)$ and $\D(?)^-_c=\D^-_{fd}(?)$ for $?\in\{A, B,C\}$. Then the statements follow directly from Theorem \ref{thm:exact seq Tbc}(1) and Corollary~\ref{cor:for Tbc}.
\end{proof}

In the corollary above, the short exact sequence in the first row of the commutative diagram is recently proved by Jin-Yang-Zhou \cite[Theorem 5.1]{JYZ23}\label{cor:smooth dg} independently by using the techniques from DG theory, here we provide a new proof and obtain more short exact sequences.

\subsection{Localization sequences in the derived categories of schemes}

In this subsection, we apply our results to derived categories of schemes.
This yields localization sequences for singularity categories,
generalizing a result of Chen \cite{Chen}.

Let $X$ be a quasicompact and quasi-separated scheme. We denote by $X\qcoh$ the quasicoherent sheaves on $X$ and $\D_{qc}(X)$ the unbounded derived category of cochain complexes of sheaves of $\mathcal{O}_X$-modules with quasicoherent cohomology. 
Let $U$ be a quasi-compact open subscheme of $X$ and write $Z=X-U$. Denote by $X\qcoh_Z$ the full subcategory of $X\qcoh$ of quasicoherent sheaves with support on $Z$. This is a Serre subcategory of $X\qcoh$ and we have a short exact sequence
$$0\lra {{X\qcoh\nolimits_Z}} {\lra } X\qcoh\lra U\qcoh\lra 0$$
of abelian categories (see \cite[Chapter III, Section 5]{g62}).
Denote by $\D_{qc,Z}(X)$  the full subcategory of  $\D_{qc}(X)$ which consists of all complexes whose cohomology is supported on $Z$.  
Due to \cite[Theorem 3.2]{N22a}, we know $\D_{qc,Z}(X)$ is weakly approximable. 
Furthermore, suppose that $X$ is noetherian. We denote by $X\coh$ the coherent sheaves on $X$ and by $X\coh_Z$ the full subcategory of $X\coh$ of coherent sheaves with support on $Z$. This is a Serre subcategory of $X\coh$ and we have a short exact sequence
$$0\lra {{X\coh\nolimits_Z}} {\lra } X\coh\lra U\coh\lra 0$$
of abelian categories (see \cite[Chapter II, Proposition 5.8 and Exercises 5.15]{ha77}).

Note that if $X$ is quasi-compact and separated, then by \cite{ajl97} the triangulated category
$\D_{qc,Z}(X)$ is equivalent to the unbounded derived category $\D(X\qcoh_Z)$ of quasi-coherent
$\mathcal{O}_X$-modules supported on $Z$.
For simplicity, we shall henceforth identify $\D_{qc,Z}(X)$ with $\D(X\qcoh_Z)$.
Under this identification, the bounded above subcategory $\D_{qc,Z}^-(X)$ coincides with the bounded above derived category $\D^-(X\qcoh_Z)$, and the bounded subcategory $\D_{qc,Z}^b(X)$ coincides with the bounded derived category $\D^b(X\qcoh_Z)$. If, in addition, $X$ is noetherian, then the bounded above category $\D_{coh,Z}^-(X)$ is equivalent
to the bounded above derived category $\D^-(X\coh_Z)$, and the bounded category
$\D_{coh,Z}^b(X)$ is equivalent to the bounded derived category $\D^b(X\coh_Z)$.

By employing the techniques developed in this paper, we provide a generalization of Chen’s result \cite[Theorem 1.3]{Chen} in statement (2) of the corollary below.

\begin{cor}\label{cor:scheme}
Let $X$ be a quasi-compact and separated scheme. Assume that $Z\subseteq X$ is a closed subset with quasi-compact complement $U$. 
Then the following statements hold. 
\begin{enumerate}
\item There is a short exact sequence of big singularity categories
 $$\D_{sg,Z}^{\rm big}(X) \longrightarrow \D_{sg}^{\rm big}(X) \longrightarrow \D_{sg}^{\rm big}(U).$$

\item
If, in addition, $X$ is noetherian, then there is a short exact sequence of singularity categories
$$\D_{sg,Z}(X)\longrightarrow \D_{sg}(X)\longrightarrow \D_{sg}(U).$$

\end{enumerate}
\end{cor}
\begin{proof}
According to \cite{J}, there exists a recollement
$$\xymatrix{
\D_{qc}(U)\ar^-{i_*=i_!}[r]
&\D_{qc}(X)\ar^-{j^!=j^*}[r]\ar^-{i^!}@/^1.2pc/[l]\ar_-{i^*}@/_1.6pc/[l]
&\D_{qc,Z}(X),\ar^-{j_*}@/^1.2pc/[l]\ar_-{j_!}@/_1.6pc/[l]}$$
that is, we have a recollement
$$\xymatrix{
\D(U\qcoh)\ar^-{i_*=i_!}[r]
&\D(X\qcoh)\ar^-{j^!=j^*}[r]\ar^-{i^!}@/^1.2pc/[l]\ar_-{i^*}@/_1.6pc/[l]
&\D({X\qcoh}_{Z}).\ar^-{j_*}@/^1.2pc/[l]\ar_-{j_!}@/_1.6pc/[l]}$$

(1) Note that there is an exact sequence of abelian categories
$$0\lra {X\qcoh}_Z\lra X\qcoh\lra U\qcoh\lra 0.$$
By \cite[Theorem 3.2]{m91}, the first row in the above recollement is restricted to a short exact sequence  
$$\D^b({X\qcoh}_{Z})\stackrel{j_!}{\longrightarrow} \D^b(X\qcoh)\stackrel{i^*}{\longrightarrow}  \D^b(U\qcoh).$$
By Theorem \ref{thm-l-sp}(1), the first row in the recollement is restricted to a short exact sequence up to direct summands
$$\xymatrix{
{\D_{qc,Z}(X)}^{{\rm sb}} \ar[r]^{j_!} 
&{\D_{qc}(X)}^{{\rm sb}} \ar[r]^{i^*}  
&{\D_{qc}(U)}^{{\rm sb}} .
}$$
Note that $\D^b(X\qcoh)\subseteq \D^-(X\qcoh)$.
The desired short exact sequence then follows from  Theorem \ref{thm-l-sp}(2).

(2) By \cite[Theorem 3.2]{m91}, the first row in the above recollement is restricted to a short exact sequence  
$$\D^b({X\coh}_{Z})\stackrel{j_!}{\longrightarrow} \D^b(X\coh)\stackrel{i^*}{\longrightarrow}  \D^b(U\coh).$$

Note that $\D^b(X\coh)\subseteq \D^-(X\coh)$. So, by Theorem \ref{thm:exact sequences of quotient}, we obtain the exact sequence.
\end{proof}

{\bf Acknowledgments:} The authors would like to thank Professor Bin Zhu for valuable discussions.
This work was supported by the National Natural Science Foundation of China (Grants 12501052, 12401038 and 12401044). The third author was also supported by the Hubei University Original Exploration Seed Fund (Grant No.260701747001).


\begin{thebibliography}{99}
\bibitem{ajl97} L. Alonso Tarr\'{i}o,  A. Jerem\'{i}as L\'{o}pez, J. Lipman, \emph{Local homology and cohomology on schemes}, Ann. Sci. \'{E}cole Norm. Sup. (4) \textbf{30}(1) (1997), 1-39.

\bibitem{ALS03}
L. Alonso Tarr\'{i}o,  A. Jerem\'{i}as L\'{o}pez, M. J. Souto Salorio,
\newblock {\em Construction of $t$-structures and equivalences of derived categories,}
\newblock Trans. Amer. Math. Soc. {\bf 355}(6) (2003), 2523-2543.

\bibitem{A09}
C. Amiot,
\newblock {\em Cluster categories for algebras of global dimension 2 and quivers with potential,}
\newblock Ann. Inst. Fourier {\bf 59}(6) (2009), 2525-2590.

\bibitem{AKLY17}
L. Angeleri H\"ugel, S. Koenig, Q. H. Liu, D. Yang,
\newblock {\em Ladders and simplicity of derived module categories,}
\newblock J. Algebra {\bf 472} (2017), 15–66.

\bibitem{BBD}
A. A. Beilinson, J. Bernstein, P. Deligne,
\newblock {\em Faisceaux perverse (French), Analysis and topology on singular spaces,}
\newblock Asterisque, vol. {\bf 100} (1982), Soc. Math. France, Paris, 5-171.

\bibitem{Bel2000}
A. Beligiannis,
{\em The homological theory of contravariantly finite subcategories: Auslander--Buchweitz contexts, Gorenstein categories and (co-)stabilization}, Comm. Algebra \textbf{28} (10) (2000), 4547-4596.

\bibitem{BR07}
A. Beligiannis and I. Reiten, \textit{Homological and homotopical aspects of torsion theories}, Mem.
Amer. Math. Soc. \textbf{188} (2007), 1-207.

\bibitem{BCPRZ24}
R. Biswas, H. X. Chen, C. Parker, K. M. Rahul, J. H. Zheng,
\newblock {\em Bounded $t$-structures, finitistic dimensions, and singularity categories of triangulated categories,}
\newblock preprint (2024), arXiv:2401.00130.

\bibitem{BV20}
M. V. Bondarko, S. V. Vostokov,
\newblock {\em On Weakly Negative Subcategories, Weight Structures, and (Weakly) Approximable Triangulated Categories,}
\newblock Lobachevskii J. Math. {\bf 41} (2020), 151–159.

\bibitem{Buch87}
R.-O. Buchsweitz, \emph{Maximal Cohen-Macaulay modules and Tate-cohomology over Gorenstein rings}, Manuscript, University of Hannover, 1987.

\bibitem{BNP18}
J. Burke, A. Neeman, B. Pauwels,
\newblock {\em Gluing approximable triangulated categories,}
\newblock Forum of Math., Sigma {\bf 11} (2023), 1–18.


\bibitem{CHNS}
A. Canonako, C. Haesemeyer, A. Neeman, P. Stellari,
\newblock{\em The passage among the subcategories of weakly approximable triangulated categories,}
\newblock preprint (2024), arXiv:2402.04605.

\bibitem{ccz25}
H. X. Chen, X. H. Chen and J. B. Zhang,
\newblock {\em Finiteness of homological dimensions in triangulated categories,} manuscript in preparation, 2026.

\bibitem{Chen}
X. W. Chen,
\newblock {\em Unifying two results of Orlov on singularity categories,}
\newblock Abh. Math. Semin. Univ. Hambg. {\bf 80} (2010), 207-212.

\bibitem{g62}
P. Gabriel, \emph{Des cat\'egories ab\'eliennes}, Bull. Soc. Math. France \textbf{90} (1962), 323-448.

\bibitem{G10}
L. Guo,
\newblock {\em Cluster tilting objects in generalized higher cluster categories,}
\newblock J. Pure Appl. Alg. {\bf 215}(9) (2011), 2055-2071.

\bibitem{ha77}
R. Hartshorne, \emph{Algebraic Geometry}, Grad. Texts in Math., vol. 52, Springer-Verlag, New York, 1977.

\bibitem{JYZ23}
H. B. Jin, D. Yang, G. D. Zhou,
\newblock {\em A localisation theorem for singularity categories of proper DG algebras,}
\newblock preprint (2023), arXiv:2302.05054.

\bibitem{J}
P. Jorgensen,
\newblock {\em A new recollement for schemes,}
\newblock Houston J. Math. {\bf 35} (2009), 1071-1077.

\bibitem{KY16}
M. Kalck, D. Yang,
\newblock{\em Relative singularity categories I: Auslander resolutions,}
\newblock Adv. Math. {\bf 301} (2016), 973–1021.


\bibitem{K96}
B. Keller,
\newblock {\em Derived categories and their uses,}
\newblock in: Handbook of Algebra, vol. 1, in: Handb. Algebr., vol. 1, Elsevier/North-Holland, Amsterdam, 1996, pp. 671–701.

\bibitem{K10}
H. Krause,
\newblock {\em Localization theory for triangulated categories,}
\newblock In T. Holm, P. J\o rgensen, R. Rouquier (Eds.), Triangulated Categories (London Mathematical Society Lecture Note Series, pp. 161-235). Cambridge: Cambridge University Press. (2010).

\bibitem{m91}
J. Miyachi,
\newblock \emph{Localization of triangulated categories and derived categories},
\newblock J. Algebra \textbf{141} (1991), 463-483.

\bibitem{N92}
A. Neeman,
\newblock {\em The connection between the K–theory localisation theorem of Thomason, Trobaugh and Yao,
and the smashing subcategories of Bousfield and Ravenel,}
\newblock Ann. Sci. \'{E}c. Norm. Sup\'{e}r. {\bf 25} (1992), 547-566.

\bibitem{N18a}
A. Neeman,
\newblock {\em Triangulated categories with a single compact generator and a Brown representability theorem,}
\newblock preprint (2018), arXiv:1804.02240.

\bibitem{N18c}
    A. Neeman,
    \newblock {\em The categories $\mathcal{T}^c$ and $\mathcal{T}_c^b$ determine each other,}
\newblock preprint (2018), arXiv:1806.06471.

\bibitem{N18d}
A. Neeman,
\newblock {\em The $t$-structures generated by objects,}
\newblock Trans. Amer. Math. Soc. {\bf 374} (2021), 8161-8175.

\bibitem{N21a}
A. Neeman,
\newblock {\em Strong generators in $\D^{\mathsf{perf}}(X)$ and $\D^{b}_{\mathsf{coh}}(X)$,}
\newblock Ann. Math. {\bf 193} (2021), 689-732.

\bibitem{N21b}
A. Neeman,
\newblock {\em Approximable triangulated categories,}
\newblock  Cont. Math. {\bf 769} (2021), 111-155.

\bibitem{N22a}
A. Neeman,
\newblock {\em Bounded $t$-structures on the category of perfect complexes,}
\newblock Acta Math. \textbf{233}(2) (2024), 239-284.

\bibitem{N01}
A. Neeman,
\newblock Triangulated categories, volume 148 of Annals of Mathematics Studies.
\newblock{\em Princeton University Press, Princeton, NJ,} 2001

\bibitem{O11}
D. O. Orlov,
\newblock{\em Formal completions and idempotent completions of triangulated categories of singularities,}
\newblock Adv. Math. {\bf 226}(1) (2011), 206-217.

\bibitem{R88}
J. Rickard,
\newblock {\em Morita theory of derived categories,}
\newblock J. Lond. Math. Soc. {\bf 39}(2) (1988), 436-456.

\bibitem{TT90}
R.~W. Thomason and T. Trobaugh,
\textit{Higher algebraic $K$-theory of schemes and of derived categories}.
\newblock In The Grothendieck Festschrift, Vol.~III, 247-435, Progr. Math. 88, Birkh\"auser, Boston, 1990.

\end{thebibliography}
\end{document}